\documentclass[12pt]{amsart}
\usepackage{amsfonts,amssymb}
\usepackage[titletoc]{appendix} 
\usepackage{mathrsfs}
\usepackage{array}
\usepackage[all,ps,cmtip,rotate]{xy} 
\usepackage[usenames,dvipsnames]{color}
\usepackage[hidelinks]{hyperref}
\usepackage{cleveref}
\usepackage{bm,mathtools}
\usepackage{enumitem}
\usepackage{amsmath}
\usepackage{mathrsfs}
\usepackage{geometry}
\usepackage{graphicx}
\usepackage{subfigure}
\usepackage{tikz-cd}
\usepackage{adjustbox}
\usepackage{appendix}
\usepackage{microtype}
\usepackage{geometry}
\newtheorem{define}{Definition}[section]
\newtheorem{thm}[define]{Theorem}
\newtheorem{prop}[define]{Proposition}

\newtheorem{rem}[define]{Remark}
\newtheorem{cor}[define]{Corollary}
\newtheorem{lem}[define]{Lemma}
\newtheorem{conj}[define]{Conjecture}

\def\CH{{\rm CH}}

\def\Z{{\mathbb{Z}}}
\def\e{{\acute{e}}}
\def\nr{\mathrm{nr}}

\def\dim{{\mathrm{dim}}}
\def\codim{{\mathrm{codim}}}
\def\P{{\mathbf{P}}}

\def\C{{\mathbb{C}}}
\def\G{{\mathbb{G}}}
\def\Q{{\mathbb{Q}}}
\def\ker{{\mathrm{ker}}}

\def\im{{\mathrm{im}}}
\def\id{{\mathrm{id}}}
\def\cl{{\mathrm{cl}}}

\newcommand{\nospacepunct}[1]{\makebox[0pt][l]{\,#1}}

\date{}
\title[On the functoriality of refined unramified cohomology]{On the functoriality of refined unramified cohomology}
\begin{document}
\author{Kees Kok}
\address{\parbox{0.9\textwidth}{KdV Institute for Mathematics, University of Amsterdam, Netherlands}}
\email{k.kok@uva.nl}
\author{Lin Zhou}
\address{\parbox{0.9\textwidth}{Institute of Algebraic Geometry, Leibniz University Hannover, Hannover, Germany.}}
\email{{zhou@math.uni-hannover.de}}
\maketitle

\begin{abstract}
In this paper, we generalise the construction of the functorial pullback of refined unramified cohomology between smooth schemes, by following the ideas of Fulton's intersection theory and Rost's cycle modules. We also define standard actions of algebraic cycles on the refined unramified cohomology groups of smooth proper schemes avoiding Chow's moving lemma, which coincide with Schreieder's constructions for smooth projective schemes. 
As applications, we prove the projective bundle and blow-up formulas for refined unramified cohomology groups and we reduce the Rost nilpotence principle in characteristic zero to statements concerning certain refined unramified cohomology groups. Moreover, we compute the refined unramified cohomology for smooth projective linear varieties and show that Rost's nilpotence principle holds for these varieties in characteristic zero.
\end{abstract}

\tableofcontents

\section{Introduction}
In \cite{schreieder2020refined}, Schreieder defines the refined unramified cohomology, which can be viewed as an intermediate between cohomology groups and algebraic cycles. Due to the nature of the definition, having a functorial pull-back between these groups is not direct. However, for smooth schemes, the (co)homology groups and Chow groups have natural pull-back maps. So it is expected that refined unramified cohomology should also have natural pull-back maps in this case. Using Chow's moving lemma, Schreieder constructs these maps for smooth quasi-projective schemes which admit smooth compactifications in \cite[Corollary 6.9]{schreieder2022moving}. In \Cref{section general pullbacks} of this paper, we propose a different construction, avoiding the moving lemma, that works for general smooth schemes, which is inspired by the construction of the pull-back map on Chow groups by Fulton in \cite{fulton2013intersection} and cycle modules by Rost in \cite{Ros96}.
\begin{thm}[{cf.~\Cref{lem-the same pullback} and \Cref{thm-functoriality}}]
Let $f\colon X\to Y$ be a morphism between smooth schemes, there is a functorial pull-back map $f^\ast\colon H^p_{q,\nr}(Y,n)\to H^p_{q,\nr}(X,n)$, which coincides with Schreieder's construction in \cite[Corollary 6.9]{schreieder2022moving} when $X$ and $Y$ are smooth quasi-projective schemes admitting a smooth compactification respectively.
\end{thm}

Having this pull-back map, we are in line to construct a correspondence action on the refined unramified cohomology groups, generalising the construction given in \cite[Corollary 6.8]{schreieder2022moving}. We do this by first constructing an external product in \Cref{subsection Cross Product}. Using this, together with the above pull-back map, we can directly define a correspondence action for smooth proper schemes.
\begin{thm}[cf.~\Cref{prop-actions of corr}, \Cref{rem-compatible action}]\label{thm-bi-additive pairing}
Let $X$, $Y$ and $Z$ be smooth proper equi-dimenisonal $k$-schemes of dimension $d_X$, $d_Y$ and $d_Z$ respectively. Then for any $c,p,q\geq 0$, there is a bi-additive pairing  
\begin{eqnarray*}
    \CH^c(X\times Y)\times H^p_{q,\nr}(Y\times Z,n)&\to& H^{p+2c-2d_Y}_{q+c-d_Y,\nr}(X\times Z,n+c-d_Y)\\
    ([\Gamma],[\alpha])&\mapsto&[\Gamma]_*[\alpha],
\end{eqnarray*}
which is compatible with the composition of correspondences. Moreover, it is compatible with the bi-additive pairing on cycle modules via the long exact sequence
\begin{eqnarray*}
    \to H^{p+q-1}_{p-2,\nr}(Y\times Z,n)\to E_2^{p,q}(Y\times Z,n)\to H^{p+q}_{p,\nr}(Y\times Z,n)\to H^{p+q}_{p-1,\nr}(Y\times Z,n)\to,
\end{eqnarray*}
constructed in \cite[\textsection 7.11]{schreieder2020refined}.
\end{thm}
We refer to \Cref{lem-concrete description of corr actions} to see that when taking $Z$ to be a point Spec$(k)$ and $X$ and $Y$ smooth projective equi-dimensional schemes, we obtain the same definition of the correspondence action as in \cite[Corollary 6.8]{schreieder2022moving}. Moreover, we define the action of cycles on refined unramified cohomology in \Cref{def-action of cycles on refined unramified cohomology}, satisfying projection formulas similar to Chow groups.
\begin{lem}[= \Cref{lem-projection formula II}]
   Let $f\colon X\to Y$ be a morphism between smooth equi-dimensional algebraic schemes. For any $\Gamma \in \CH^c(Y)$ and $\alpha\in H^i_{j,\nr}(Y,n)$, we then have $f^*(\Gamma\cdot\alpha)=(f^*\Gamma)\cdot(f^*\alpha)$. Moreover, if $f$ is proper, then we have the projection formulas:
   \begin{itemize}
       \item[(i)] for any $\Gamma\in \CH^c(X)$ and $\alpha\in H^i_{j,\nr}(Y,n)$, we have $f_*(\Gamma\cdot f^*(\alpha))=f_*(\Gamma)\cdot\alpha$;
       \item[(ii)] for any $\Gamma\in \CH^c(Y)$ and $\alpha\in H^i_{j,\nr}(X,n)$, we have $f_*(f^*(\Gamma)\cdot \alpha)=\Gamma\cdot f_*(\alpha)$. 
   \end{itemize}  
\end{lem}

As a direct consequence of the natural correspondence action above, we obtain a projective bundle and blow-up formula for the refined unramified cohomology groups in \Cref{application formulas}.
\begin{prop}[cf.~\Cref{prop-projective bundle formula}]
Let $X$ be a smooth scheme and $E\to X$ be a vector bundle of rank $r+1$. There is a canonical isomorphism $ \bigoplus_{k=0}^{r} H^{p-2k}_{q-k,\nr}(X,n-k)\xrightarrow{\sim}H^p_{q,\nr}(\P(E),n)$.
\end{prop}
\begin{prop}[cf.~\Cref{prop-blow-up bundle}]
Let $Y\subset X$ be a closed embedding of smooth schemes with codimension $r+1$ and let $\tau\colon\tilde X\to X$ be the blow-up of $X$ with centre $Y$. There is a canonical isomorphism
$H^p_{q,\nr}(X,n)\oplus\bigoplus_{k=0}^{r-1}H^{p-2k-2}_{q-k-1,\nr}(Y,n-k-1)\xrightarrow{\sim}H^p_{q,\nr}(\tilde X,n)$.  
\end{prop}
As a second application, we reduce Rost's nilpotence principle  \Cref{conj-Rost} (RNP for short) in characteristic zero to the similar statement concerning the refined unramified cohomology groups in \Cref{application RNP}. The RNP for smooth projective quadrics (see \cite{rost1998motive}) plays an important role in  Voevodsky’s proof of the Milnor conjecture \cite{voevodsky2011motivic}. For perfect fields with $\dim(X)\leq 2$, this is proven by \cite{rosenschon2018rost} and \cite{diaz2019rost}, which generalise the works of Gille \cite{gille2010rost,gille2014chow}. Also RNP is known for isotropic projective homogeneous varieties for a semisimple algebraic group \cite{MR2110630} and for three-dimensional smooth projective geometrically integral schemes over a field of characteristic zero which are birationally isomorphic to a toric model \cite{gille2018rost}. However, RNP is wide open for $\dim(X)\geq 3$.

\begin{prop}[= \Cref{prop-RNP}]
   Let $X$ be a smooth projective equi-dimensional $d$-dimensional scheme defined over a field $k$ of characteristic zero. Let $\Gamma\in {\rm End}_k(X)$ such that $\Gamma_E=0$ for some finite Galois field extension $E/k$. The following are equivalent:
    \begin{itemize}
        \item[(1)] $\Gamma$ is nilpotent as a correspondence;
        \item[(2)] the action of the correspondence $\Gamma$ is nilpotent on the refined unramified cohomology groups $H^{2d-2}_{d-3,\nr}(X\times _kX,\Q_{\ell}/\Z_{\ell}(d))$ for any prime $\ell$, and the nilpotence index of $[\Gamma]_*$ on $H^{2d-2}_{d-3,\nr}(X\times _kX,\Q_{\ell}/\Z_{\ell}(d))$ has a common upper bound for all $\ell$;
        \item[(3)] the action of the correspondence $\Gamma$ is nilpotent on the refined unramified cohomology groups $H^{2d-1}_{d-2,\nr}(X\times _kX,\Z_{\ell}(d))$ for any prime $\ell$, and the nilpotence index of $[\Gamma]_*$ on $H^{2d-1}_{d-2,\nr}(X\times _kX,\Z_{\ell}(d))$ has a common upper bound for all $\ell$.
      \end{itemize}
Here $H^i(*,\Z_{\ell}(n))$ is Jannsen's continuous \'etale cohomology (see \cite{jannsen1988continuous}) with coefficients $(\mu_{\ell^{r}}^{\otimes  n})_r$.
\end{prop}
Finally, we compute the refined unramified cohomology for smooth projective linear varieties in \Cref{section-linear vars}.
\begin{prop}[cf.~\Cref{prop-surjective}, \Cref{lem-computation for algebraically closed field}]
    Let $X$ be a smooth projective linear variety defined over a field $k$. Then the restriction map $H^{p}(X,M(b))\to H^{p}_{q,\nr}(X,M(b))$ is surjective for all $p,q\geq 0$ and all $b$, where  $m$ is invertible in $k$ and $M\coloneqq \Z/m$. In particular, RNP holds for linear varieties over characteristic zero.
    
    In addition, if $k$ is an algebraically closed field 
   \begin{equation*}
  H^p_{q,\nr}(X,M(b))=\left\{\begin{array}{ll}
      H^p(X,M(b)) & \text{if } p\leq 2q \text{ and } p \text{ is even; } \\
      0 & \text{otherwise.}
  \end{array} \right.
\end{equation*}
In particular, the higher cycle class map $\CH^b(X,n;M)\to H^{2b-n}(X,M(b))$ is an isomorphism for $b,n\geq 0$.
\end{prop}
\subsection{Acknowledgements}
The authors are very grateful to Stefan Schreieder for his comments on the first draft and for his help in improving Section 4.2. The paper is written during the first author's PhD which was supported by the Dutch Research Council (NWO) Vidi grant 016.Vidi.189.015. The second author has received funding from the European Research Council under the European Union's Horizon 2020 research and innovation program under grant agreement
No 948066 (ERC-StG RationAlgic).
\subsection{Notations and Conventions}
\begin{itemize}
    \item  All schemes are separated, Noetherian and of finite type over a field $k$. 
    \item A variety is an integral scheme.
    \item Generally we denote $N(X,Y)\coloneqq N_{Y}X$ for the normal bundle of a regular embedding $Y\hookrightarrow X$.
    \item We write $\G_m\coloneqq \mathbb{A}^1_k\setminus \{0\}$ for convenience, without  making use of its group structure.
\end{itemize}

\section{Pullback on Refined Unramified Cohomology}
In this section, we show that refined unramified cohomology has functorial pullbacks between smooth schemes and provide concrete descriptions of this pullback under certain specific conditions. First of all, let us recall some elementary facts.
\subsection{Borel--Moore Homology and \'Etale Cohomology with Supports}
Let $k$ be a field, and let $\mathcal{C}$ be a category of $k$-schemes containing all quasi-projective  $k$-schemes such that every closed (or open) immersion is a morphism in $\mathcal{C}$.
\begin{define}[cf. {\cite[Section 2]{BO74}}]\label{def-BM homology and etale cohomology with supports}
Let $\ell$ be a positive integer invertible in $k$, and let $\pi_X\colon X\to {\rm Spec}(k)$ denote the structure map of an object $X\in \mathcal{C}$. Let $Z\subset X$ be a closed subscheme of $X$, and let $\mathcal{F}$ be an $\ell^{\infty}$-torsion  \'etale sheaf on ${\rm Spec}(k)$. We define 
\[
    H_i(X,n)\coloneqq H^{-i}(X_{\e t},\pi^!_X\mathcal{F}(n-d));
\]
\[
    H^i_Z(X,n)\coloneqq H^i_Z(X_{\e t},\pi^*_X\mathcal{F}(n)).
\]
For $Z=X$, we write $H^*(X,n)\coloneqq H^*_X(X,n)$. Moreover, for $k=\C$, let $A$ be an abelian group. Then we define $H_i(X,n)$ $($resp. $H^i_Z(X,n)$ $)$ as the singular homology group $($resp. singular cohomology group with supports$)$ $H_i^{BM}(X(\C),A)$ $($resp. $H^i_{Z(\C)}(X(\C),A)$ $)$.
\end{define}
The group $H^i_Z(X,n)$ is a twisted cohomology with supports in the sense of \cite[Section 1]{BO74}; see also \cite{viale1997ℋ︁}, \cite[Definition 4.3; Definition 4.4]{kerz2012cohomological} and \cite[Section 3]{schreieder2022moving}. In particular, $H_i(X,n)$ is a homology theory with duality $H^*_{Z}(X,n)$.

Let $F_jX\coloneqq\{x\in X|\codim_X(x)\leq j\}$, where $\codim_X(x)\coloneqq\dim(X)-\dim(x)$. We define 
\[
H^i(F_jX,n)\coloneqq\varinjlim_{F_jX\subset U\subset X}H^i(U,n),
\]
where $U$ runs through all open subsets of $X$ containing $F_jX$. Moreover, we define \emph{the $j$-th refined unramified cohomology} of $X$ with values in $A(n)$ by
\begin{equation*}
    H^i_{j,\nr}(X,n)\coloneqq\im (H^i(F_{j+1}X,n)\to H^i(F_jX,n)).
\end{equation*}
Note here that if $X$ is a smooth equi-dimensional algebraic scheme, this definition is the same as \cite[Definition 5.1]{schreieder2020refined} and \cite[Section 2.4]{schreieder2020infinite} making use of \emph{Gabber's purity isomorphism} here (cf.~\cite[Lemma 6.5]{schreieder2020refined}). In particular, we still have a canonical long exact sequence for smooth $X$ by \cite[Proposition 7.35]{schreieder2020refined}.
\begin{prop}[{\cite[Proposition 7.35]{schreieder2020refined}}]\label{prop-schreieder's prop 7.35}
 If $X$ is a smooth equi-dimensional algebraic scheme over a field $k$, there is a canonical long exact sequence   
 \begin{equation}\label{equ-the canonical long exact sequence}
  \cdots\to H_{j-1,\nr}^{i+2j-1}(X,n)\xrightarrow{\rm restr.} H_{j-2,\nr}^{i+2j-1}(X,n)\xrightarrow{\partial} E_2^{j,i+j}(X,n)\xrightarrow{\iota_*} H_{j,\nr}^{i+2j}(X,n)\to\cdots,  
\end{equation}
where $E_2^{j,i+j}(X,n)\cong H^j(X,\mathcal{H}^{i+j}_X(n))$ and $\mathcal{H}^{i+j}_X(n)$ is the Zariski sheaf assocatied with the Zariski presheaf $U\mapsto H^{i+j}_{\e t}(U,n)$. 
\end{prop}
More specifically, (\ref{equ-the canonical long exact sequence}) is induced by the canonical long exact sequence
\begin{equation*}
    \cdots\to H^i(F_{j}X,n)\xrightarrow{\rm restr.} H^i(F_{j-1}X,n)\overset{\partial}{\to}\bigoplus_{x\in X^{(j)}}  H^{i+1-2j}(x,n-j)\overset{\iota_\ast}{\to}\cdots,
\end{equation*}
where $X^{(j)}$ is the set of codimension-$j$ points of $X$ and $\iota_*$ is the \emph{Gysin homomorphism} for smooth pairs. Moreover, the group $\bigoplus_{x\in X^{(j)}} H^{i+1-2j}(x,n-j)$ is part of a \emph{Rost's cycle module complex} (cf.~\cite{Ros96}).

Due to the work of Schreieder, for smooth projective $X$ and $Y$, there exists an action of correspondence on refined unramified cohomology.
\begin{prop}[{\cite[Corollary 1.7]{schreieder2022moving}}]\label{prop-correspondences actions}
 Let $X$ and $Y$ be smooth projective equi-dimensional schemes over a field $k$ with $d_X\coloneqq\dim(X)$. For all $c,i,j\geq 0$, there is a natural bi-additive pairing
 $$\CH^c(X\times Y)\times H^i_{j,\nr}(X,n)\to H^{i+2c-2d_X}_{j+c-d_X,\nr}(Y,n+c-d_X),$$
 which is functorial with respect to the composition of correspondences.
\end{prop} 
An explicit description of this action is given in \cite[Corollary 6.8]{schreieder2022moving}.

\subsection{General Pullbacks along Smooth Schemes}\label{section general pullbacks}
In this subsection, we construct the functorial pullbacks of refined unramified cohomology between smooth algebraic schemes. First, if $f\colon X\to Y$ is a flat morphism between arbitrary algebraic schemes, we can define $f^*\colon H^i_{j,\nr}(Y,n)\to H^i_{j,\nr}(X,n)$ naturally as $F_jX\subset f^{-1}(U)$ for any open subset $F_jY\subset U\subset Y$. Now, suppose $f\colon X\to Y$ is a morphism of smooth $k$-schemes. Then $f$ factors as $f\colon X\xrightarrow{i}X\times Y\xrightarrow{p} Y$, where $i$ is a closed immersion and $p$ is the (smooth) projection. Naturally, we want to define $f^*\coloneqq i^*\circ p^*\colon H^{i}_{j,\nr}(Y,n)\to H^{i}_{j,\nr}(X,n)$, where we can define $p^*$ as above, so we are left to define the pullback along a closed immersion $i\colon X\hookrightarrow X\times Y$ of smooth $k$-schemes. The following construction is inspired by Fulton \cite{fulton2013intersection}.
\begin{define}\label{def-pullback along lci}
We define $i^*$ as the following composition
\begin{align*}
    \begin{split}
       H^i_{j,\nr}(X\times Y,n)\xrightarrow{\pi^*}H^i_{j,\nr}(X\times Y\times(\mathbb{A}^1\setminus \{0\}),n)\xrightarrow{\{t\}} H^{i+1}_{j,\nr}(X\times Y\times(\mathbb{A}^1\setminus \{0\}),n+1)\\
       \xrightarrow{-\delta^{i+1}_j}H^i_{j,\nr}(N_X(X\times Y),n)\cong H^i_{j,\nr}(X,n),
    \end{split}
\end{align*}
where
\begin{enumerate}
    \item $\pi\colon X\times Y\times(\mathbb{A}^1\setminus \{0\})\to X\times Y$ is the projection;
    \item $\{t\}$ is induced by \emph{multiplication with units} with $\mathbb{A}^1=\operatorname{Spec} k[t]$ (see below for a precise definition);
    \item $\delta^{i+1}_{j}$ is the boundary morphism for $N_X(X\times Y)\hookrightarrow D(X\times Y,X)$ in \cite[Lemma~3.1]{kok2023higher}, where $D(X\times Y,X)$ is the deformation space of $X\subset X\times Y$ and $N_X(X\times Y)$ is the normal bundle of $X\subset X\times Y$;
    \item  the last isomorphism is given by the inverse of $\mathbb{A}^1$-homotopy invariance in \cite[Proposition 3.10]{kok2023higher}, as $N_X(X\times Y)\cong f^*TY$ is a vector bundle over $X$.
\end{enumerate}
Moreover, we denote $J(X\times Y,Y)\coloneqq -\delta\circ \{t\}\circ \pi^*\colon  H^i_{j,\nr}(X\times Y,n)\to H^i_{j,\nr}(N_X(X\times Y),n)$.
\end{define}
Let us elaborate a bit on this definition. The element $t$ belongs to $H^0(\mathbb{A}^1\setminus \{0\},\mathbb{G}_m)$ and further can be viewed as an element of the group 
\[
\im(H^0(X\times Y\times(\mathbb{A}^1\setminus \{0\}),\mathbb{G}_m)\to H^1(X\times Y\times(\mathbb{A}^1\setminus \{0\}),1))
\]
via the Kummer sequence, which we also denote by $t$. There is a cup product 
 \begin{equation}\label{equ-cap product}
     \cup \colon H^p(X,m)\otimes H^q(X,n)\to H^{p+q}(X,m+n)
 \end{equation}
for any $X$. Let $X$ be an open subset $U$ of $X\times Y\times(\mathbb{A}^1\setminus \{0\})$ and $t$ be the image of $t$ under the restriction map $H^1(X\times Y\times(\mathbb{A}^1\setminus \{0\}),1)\to H^1(U,1)$. Then $\{t\}$ is exactly the morphism of cup product of $t\cup$.
\begin{lem}\label{lem-the same pullback}
Using the same notations as above, suppose that $X$ and $Y$ admit smooth projective compactifications and are equi-dimensional. This pullback $f^*$ is the same as the pullback in \cite[Corollary 6.9]{schreieder2022moving}. 
\end{lem}
\begin{proof}
Let us start with recalling the construction in \cite[Corollary 6.9]{schreieder2022moving} of the pullback of an $\alpha\in H^i_{j,\nr}(Y,n)$. The key is \cite[Corollary 6.5]{schreieder2022moving}, which tells us that there exists an open subset $U$ with $F_{j+1}Y \subset U$ such that $\alpha \in H^i(U,n)$ and $F_{j+1}X\subset f^{-1}(U)\eqqcolon V$, so we can define $f^*([\alpha])=[{(f|_{f^{-1}(U)})}^*(\alpha)] \in H^i_{j,\nr}(X,n)$. Taking such $U$, then we have
\begin{itemize}
    \item $p^*[\alpha]=[p^*(\alpha)]$ with $p^*(\alpha)\in H^i(X\times U,n)$ and $F_{j+1}X\subset i^{-1}(X\times U)$;
    \item $\pi^*\circ p^*[\alpha]=[\pi^*\circ p^*(\alpha)]$ with $\pi^*\circ p^*(\alpha)\in H^i(X\times U\times (\mathbb{A}^1\setminus \{0\}),n)$;
    \item  $\{t\}\circ\pi^*\circ p^*[\alpha]=[\{t\}\cup\pi^*\circ p^*(\alpha)]$ which is also defined over $\mathcal{X}\coloneqq X\times U\times (\mathbb{A}^1\setminus \{0\})$;
    \item $-\delta^{i+1}_j\circ \{t\}\circ\pi^*\circ p^*[\alpha]=[-\partial\circ\{t\}\cup\pi^*\circ p^*(\alpha)]$, where 
    \[\partial\colon H^{i+1}(\mathcal{X},n
    +1)\to H^i(N_X(X\times Y)\setminus(N_X(X\times Y)\cap W),n)\]
    and $W$ is the closure of $X\times (Y\setminus U)\times (\mathbb{A}^1\setminus\{0\})$ in $D(X\times Y,X)$.
\end{itemize}
Note that we have the following commutative diagram
\begin{equation*}
    \xymatrix{X\times U\times (\mathbb{A}^1\setminus\{0\})\ar@{^(->}[r]^{j''}\ar@{=}[d]&D(X\times Y,X)|_{X\times U\times \mathbb{A}^1}\ar@{^(->}[d]^{\rm open}&&N_X(X\times Y)|_{f^{-1}(U)}\ar@{_(->}[ll]_{i''}\ar@{^(->}[d]^{\rm open}\\
   X\times U\times (\mathbb{A}^1\setminus\{0\})\ar@{^(->}[r]^{j'}&D(X\times Y,X)\setminus W&&N_X(X\times Y)\setminus W,\ar@{_(->}[ll]_{i'}
   }
\end{equation*}
where $i'$ (resp. $i''$) is a closed immersion with the open complement $j'$ (resp. $j''$). Since $F_{j+1}X\subset f^{-1}(U)$, then $F_{j+1}N_X(X\times Y)\subset N_X(X\times Y)|_{f^{-1}(U)}$, and $-\delta^{i+1}_j\circ \{t\}\circ\pi^*\circ p^*[\alpha]$ can be represented by $[-\partial\circ\{t\}\cup\pi^*\circ p^*(\alpha)]$ where $\partial\colon H^{i+1}(\mathcal{X},n
    +1)\to H^i(N_X(X\times Y)|_{f^{-1}(U)},n)$ with $-\partial\circ\{t\}\cup\pi^*\circ p^*(\alpha)\in H^i(N_X(X\times Y)|_{f^{-1}(U)},n)$.
Now consider the following commutative diagram
\begin{equation*}
   \xymatrix{
   X\times U\times (\mathbb{A}^1\setminus\{0\})\ar@{^(->}[r]^{j''}\ar[dr]^{\pi}&D(X\times Y,X)|_{X\times U\times \mathbb{A}^1}\ar[d]^{h}&&N_X(X\times Y)|_{f^{-1}(U)}\ar@{_(->}[ll]_{i''}\ar[d]^{h'}\\
   &X\times U\ar[d]^{p}&&f^{-1}(U)\ar@{_(->}[ll]_{i}\ar[dll]^{f|_{f^{-1}(U)}}\\
   &U&&,
   } 
\end{equation*}
where $\pi$, $h$, $h'$ and $p$ are all projections. Hence $$\pi^*\circ p^*(\alpha)={j''}^*\circ h^*\circ p^*(\alpha)=h^*\circ p^*(\alpha)|_{X\times U\times (\mathbb{A}^1\setminus\{0\})},$$ so that $-\partial(\{t\}\cup(\pi^*\circ p^*(\alpha)))=h^*\circ p^*(\alpha)|_{N_X(X\times Y)|_{f^{-1}(U)}}={i''}^*\circ h^*\circ p^*(\alpha)={h'}^*\circ i^*\circ p^*(\alpha)$. Therefore, $-\delta^{i+1}_j\circ \{t\}\circ\pi^*\circ p^*[\alpha]=[{h'}^*\circ i^*\circ p^*(\alpha)]=[{h'}^*\circ (f|_{f^{-1}(U)})^*(\alpha)]$, which composed with the inverse of ${h'}^*$ is exactly $[(f|_{f^{-1}(U)})^*(\alpha)]$.
\end{proof}
\begin{lem}
    If $f\colon X\to Y$ is a flat morphism with $X$ and $Y$ smooth, then $f^*\colon H^i_{j,\nr}(Y,n)\to H^i_{j,\nr}(X,n)$ is exactly the usual pull-back of refined unramified cohomology. 
\end{lem}
\begin{proof}
    This is essentially the same as the proof of \Cref{lem-the same pullback}.
\end{proof}

\begin{rem}\label{pull-back ordinary cohomology}
{\rm We note that the proof of \Cref{lem-the same pullback} also gives that for large $j$, this is precisey the pull-back map on ordinary cohomology.}
\end{rem}

\begin{lem}\label{lem-functorility for redular embedding}
    If $f\colon X\hookrightarrow Y$ is a regular embedding with smooth $X$ and $Y$, then $f^*$ is well defined, i.e., $f^*=i^*\circ p^*$ with $i\colon X\hookrightarrow X\times Y$ and $p\colon  X\times Y\to Y$, where $f^*$ is defined by \Cref{def-pullback along lci} directly. 
\end{lem}
\begin{proof}
    Note that $q\colon  D(X\times Y,X)\to D(Y,X)$ is flat and this morphism induces a projection of vector bundles over $X$, i.e., $N_X(X\times Y)\to N_XY$. This lemma is given by the commutative \ref{diagram-functor for regular embedding}, where ${'\pi}\colon  Y\times \mathbb{G}_m\to Y$ is a projection and all horizontal arrows are induced by the pull-back along $q$ which are all flat pull-backs.
\end{proof}
\begin{equation*}\label{diagram-functor for regular embedding}
        \xymatrix{
        H^i_{j,\nr}(Y,n)\ar@{=}[r]\ar[d]^{p^*}&H^i_{j,\nr}(Y,n)\ar[dd]^{'\pi^*}\\
         H^i_{j,\nr}(X\times Y,n)\ar[d]^{\pi^*}&&\\
         H^i_{j,\nr}(X\times Y\times \mathbb{G}_m,n)\ar[d]^{\{t\}}&H^i_{j,\nr}(Y\times \mathbb{G}_m,n)\ar[d]^{\{t\}}\ar[l]_{\hspace{8mm}q^*}\\
         H^{i+1}_{j,\nr}(X\times Y\times \mathbb{G}_m,n+1)\ar[d]^{-\delta}&H^{i+1}_{j,\nr}(Y\times \mathbb{G}_m,n+1)\ar[d]^{-\delta}\ar[l]_{\hspace{8mm}q^*}\\
         H^i_{j,\nr}(N_X(X\times Y),n)&H^i_{j,\nr}(N_XY,n)\ar[l]_{\hspace{8mm}q^*}
        }\tag{\text{Diagram \ref{lem-functorility for redular embedding}}}
    \end{equation*}
    
In order to prove \Cref{thm-functoriality}, the functoriality of pull-backs between smooth schemes, let us first list some facts.
\begin{lem}[cf.~{\cite[Lemma 11.3]{Ros96}}]\label{lem-lem11.3}
    Suppose $Y\hookrightarrow X$ is a regular embedding and $g\colon  V\to X$ is flat of constant relative dimension. Let 
    \begin{equation*}
        N(g)\colon  N(V,Y\times_XV)=N(X,Y)\times_XV\to N(X,Y)
    \end{equation*}
    be the flat projection. Then $$J(V,Y\times_XV)\circ g^*=N(g)^*\circ J(X,Y)\colon  H^i_{j,\nr}(X,n)\to H^i_{j,\nr}(N(X,Y)\times _XV,n).$$
\end{lem}
\begin{proof}
    Note that $D(V,Y\times_X V)\to D(X,Y)$ is flat. We can get this lemma by following the line of reasoning in the proof of \Cref{lem-functorility for redular embedding}.
\end{proof}
\begin{lem}[cf.~{\cite[Lemma 11.4]{Ros96}}]\label{lem-lem11.4}
  Let $U\hookrightarrow V$ be a closed immersion and let $p\colon  V\to W$ be a flat morphism. Suppose that the composite $q\colon  N_UV\to U\to W$ is also flat of the same relative dimension as $p$. Then $$J(V,U)\circ p^*=q^*\colon  H^i_{j,\nr}(W,n)\to H^i_{j,\nr}(N_UV,n).$$  
\end{lem}
\begin{proof}
See the proof of \cite[Lemma 11.4]{Ros96}. Everything in the proof remains valid if we replace the cycle modules with the refined unramified cohomology.
\end{proof}
\begin{lem}[cf.~{\cite[Lemma 11.7]{Ros96}}]\label{lem-lem11.7}
 Let $Z\xrightarrow{g}Y\xrightarrow{f}X$ be regular embeddings. Then
 \begin{equation*}
     J(N_YX,N_YX|_Z)\circ J(X,Y)=J(N_ZX,N_ZY)\circ J(X,Z)\colon  H^i_{j,\nr}(X,n)\to H^i_{j,\nr}(N(N_ZX,N_ZY),n).
 \end{equation*}
\end{lem}
\begin{proof}
 Let us consider the double deformation space $D(X,Y,Z)\coloneqq \overline{D}(X,Y,Z)$; see \cite[Section~10]{Ros96}. Then $D(X,Y,Z)$ is flat over $\mathbb{A}^2$ such that 
 \begin{eqnarray*}
   D|_{\mathbb{A}^1\times \mathbb{G}_m}&=&D(X,Y)\times \mathbb{G}_m\\
   D|_{\mathbb{G}_m\times \mathbb{A}^1}&=&\mathbb{G}_m \times D(X,Z)\\
   D|_{(0,0)}&=&N(N_ZX,N_ZY).
 \end{eqnarray*}
  Suppose that $\mathbb{A}^2=\operatorname{Spec}k[t,s]$ with coordinate $(t,s)$ and let $\delta_1^3$, $\delta_0^1$, $\delta_2^3$, $\delta_0^2$ be the boundary maps of refined unramified cohomology associated with the closed immersions 
  \[
  D|_{0\times \mathbb{G}_m}\hookrightarrow D|_{\mathbb{A}^1\times \G_m},\hspace{2mm} D_{(0,0)}\hookrightarrow D|_{0\times \mathbb{A}^1}, \hspace{2mm} D|_{\G_m\times 0}\hookrightarrow D|_{\G_m\times \mathbb{A}^1},\hspace{2mm} D|_{(0,0)}\hookrightarrow D|_{\mathbb{A}^1\times 0}
  \]
respectively. Let $\pi_{1,i}\colon  X\times \G_m\times \G_m\to X\times \G_m$ be the projections to the $(1,i)$-factors. Let $\pi\colon  X\times \G_m\to X$ and $q\colon  X\times \G_m\times \G_m\to X$ be the projections. Then by definition, we have
\begin{eqnarray*}
   J(N_YX,N_YX|_Z)\circ J(X,Y)&=&-\delta_0^1\circ \{s\}\circ (N_YX\times \G_m\to N_YX)^*\circ J(X,Y)\\
   \text{(by \Cref{lem-lem11.3})}&=&-\delta_0^1\circ \{s\}\circ J(X\times \G_m,Y\times \G_m)\circ \pi^*\\
   &=&-\delta_0^1\circ \{s\}\circ (-\delta^3_1)\circ \{t\}\circ \pi_{1,3}^*\circ \pi^*\\
   &=&-\delta_0^1\circ \{s\}\circ (-\delta^3_1)\circ \{t\}\circ q^*;\\
   J(N_ZX,N_ZY)\circ J(X,Z)&=& -\delta_0^2\circ \{t\}\circ (N_ZX\times \G_m\to N_ZX)^*\circ J(X,Z)\\
   \text{(by \Cref{lem-lem11.3})}&=&-\delta_0^2\circ \{t\}\circ J(X\times \G_m,Z\times \G_m)\circ \pi^*\\
   &=&-\delta_0^2\circ \{t\}\circ (-\delta_2^3)\circ\{s\} \circ \pi_{1,2}^* \circ \pi^*\\
   &=&-\delta_0^2\circ \{t\}\circ (-\delta_2^3)\circ\{s\} \circ q^*.
\end{eqnarray*}
Thus it suffices to show that $\delta_0^1\circ \{s\}\circ \delta^3_1\circ \{t\}=\delta_0^2\circ \{t\}\circ \delta_2^3\circ\{s\}$. Indeed, we have 
\begin{eqnarray*}
    \delta_0^1\circ \{s\}\circ \delta^3_1\circ \{t\}&\overset{(1)}{=}&-\delta_0^1\circ \delta^3_1\circ \{s\}\circ \{t\}\\
    &\overset{(2)}{=}& \delta_0^2\circ\delta_2^3\circ \{s\}\circ \{t\}\\
    &\overset{(3)}{=}& -\delta_0^2\circ\delta_2^3\circ \{t\} \circ\{s\}\\
    &\overset{(4)}{=}&\delta_0^2\circ \{t\}\circ \delta_2^3\circ\{s\}.
\end{eqnarray*}
Here (1) and (4) are induced by the equality $\partial(t\cup\alpha)=-t\cup \partial(\alpha)$ for any $\alpha\in H^i(U,n)$ and $t\in \Gamma(U,\mathcal{O}^*_U)$, (2) is given by \cite[Lemma 2.4]{JS03}, which is the same as the argument in the proof of \Cref{prop-exact sequence is compatible with pullback}, and (3) is induced by $t\cup s=-s\cup t$ for any $s,t\in H^1(U,1)$.
\end{proof}
\begin{thm}\label{thm-functoriality}
    For morphisms $g\colon  Z\to Y$ and $f\colon  Y\to X$ with smooth $X$, $Y$ and $Z$, one has $(f\circ g)^*=g^*\circ f^*$. In particular, our definition is well-defined.
\end{thm}
\begin{proof}
The proof is inspired by \cite[Section 13]{Ros96}.  Let $h\coloneqq f\circ g\colon  Z\to X$. First, note that the statement holds for flat $f$ and $g$, as in this case, the pull-backs are exactly the usual pull-backs.

\emph{Step 1: reduction to the case where $f$ and $g$ are both regular embeddings.}

Consider the following commutative diagram
\begin{equation*}
\xymatrix{
Z\ar[r]^{i_1}&Z\times Y\ar[r]^{i_2}\ar[d]^{p_1}&Z\times Y\times X\ar[d]^{p_2}\\
&Y\ar[r]^{i_3}&Y\times X\ar[d]^{p_3}\\
&&X,
}
\end{equation*}
where $p_j(j=1,2,3)$ is the projection, $i_1\coloneqq \Gamma_g, i_3\coloneqq \Gamma_f$ and $i_2\coloneqq \id\times \Gamma_f$ are regular embeddings. We are going to show that $p_1^*\circ i_3^*=i_2^*\circ p_2^*$. If it holds and $(f\circ g)^*=g^*\circ f^*$ holds for any regular embedding, then $g^*\circ f^*=i_1^*\circ p_1^*\circ i_3^*\circ p_3^*=i_1^*\circ i_2^*\circ p_2^*\circ p_3^*=(i_1\circ i_2)^*\circ (p_3\circ p_2)^*$. Note that $h$ also factors as $Z\xrightarrow[i_2\circ i_1]{(\id,g,h)}Z\times Y\times X\xrightarrow{p_2}Y\times X\xrightarrow{p_3}X$, and one can show that $h^*=(i_2\circ i_1)^*\circ (p_3\circ p_2)^*$ by applying \Cref{lem-functorility for redular embedding}. Thus one can get $h^*=g^*\circ f^*$.

The equation $p_1^*\circ i_3^*=i_2^*\circ p_2^*$ comes from the following commutative diagram
\begin{equation*}
\xymatrix{H^i_{j,\nr}(Y\times X,n)\ar[d]^{p_2^*}&H^i_{j,\nr}(Y\times X,n)\ar[dd]^{\pi^*}\ar@{=}[l]\\
H^i_{j,\nr}(Z\times Y\times X,n)\ar[d]^{\pi^*_0}&\\
H^i_{j,\nr}(Z\times Y\times X\times \mathbb{G}_m,n)\ar[d]^{\{t\}}&H^i_{j,\nr}(Y\times X\times \mathbb{G}_m,n)\ar[d]^{\{t\}}\ar[l]_{\hspace{4mm}(p_2\times \id)^*}\\
H^{i+1}_{j,\nr}(Z\times Y\times X\times \mathbb{G}_m,n+1)\ar[d]^{-\delta}&H^{i+1}_{j,\nr}(Y\times X\times \mathbb{G}_m,n+1)\ar[d]^{-\delta}\ar[l]_{\hspace{4mm}(p_2\times \id)^*}\\
H^i_{j,\nr}(N_{Z\times Y}(Z\times Y\times X),n)&H^i_{j,\nr}(N_Y(Y\times X),n)\ar[l]_{\hspace{4mm}(p_2\times \id)^* }\\
H^i_{j,\nr}(Z\times Y,n)\ar[u]_{\sim}&H^i_{j\nr}(Y,n)\ar[l]_{p_1^*}\ar[u]_{\sim},
    }
\end{equation*}
where $\pi\colon  Y\times X\times \mathbb{G}_m\to Y\times X$, $\pi_0\colon  Z\times Y\times X\times \mathbb{G}_m\to Z\times Y\times X$ and the left $-\delta$ (\emph{resp.} the right $-\delta$) is induced by the deformation space $D(Z\times Y\times X,Z\times Y)$ (\emph{resp.} $D(Y\times X,Y)$). The commutativity of the above diagram is given by the flatness of $D(Z\times Y\times X,Z\times Y)\to D(Y\times X,Y)$ as the proof of \cite[Lemma 11.3]{Ros96}, and note that all pull-backs in this diagram are flat pull-backs.

\emph{Step 2: the case that $f$ and $g$ are regular embeddings .}

This is essentially the proof of \cite[Theorem 13.1]{Ros96}. For any vector bundle $E\to W$, we write $r(E,W)$ as the inverse of the $\mathbb{A}^1$-homotopy invariant $H^i_{j,\nr}(W,n)\xrightarrow{\cong}H^i_{j,\nr}(E,n)$. Then we have
\begin{eqnarray*}
    g^*\circ f^*&=&r(N_ZY,Z)\circ J(Y,Z)\circ r(N_YX,Y)\circ J(X,Y)\\
    &=&r(N(N_YX,N_YX|_Z),Z)\circ (N(N_YX,N_YX|_Z)\to N(Y,Z))^*\circ J(Y,Z)\circ r(N_YX,Y)\circ J(X,Y)\\
    &\overset{(1)}{=}&r(N(N_YX,N_YX|_Z),Z)\circ J(N_YX,N_YX|_Z)\circ (N(Y,X)\to Y)^*\circ r(N_YX,Y)\circ J(X,Y)\\
    &=&r(N(N_YX,N_YX|_Z),Z)\circ J(N_YX,N_YX|_Z)\circ J(X,Y)\\
    &\overset{(2)}{=}&r(N(N_YX,N_YX|_Z),Z)\circ J(N_ZX,N_ZY)\circ J(X,Z)\\
    &=&r(N(N_YX,N_YX|_Z),Z)\circ J(N_ZX,N_ZY)\circ(N_ZX\to Z)^*\circ r(N_ZX,Z)\circ J(X,Z)\\
    &\overset{(3)}{=}&r(N(N_YX,N_YX|_Z),Z)\circ (N(N_ZX,N_ZY)\to Z)^*\circ r(N_ZX,Z)\circ J(X,Z)\\
    &=&r(N(N_YX,N_YX|_Z),Z)\circ J(X,Z)\\
    &=&h^*,
\end{eqnarray*}
where (1) is given by \Cref{lem-lem11.3}, (2) is given by \Cref{lem-lem11.7} and (3) is given by \Cref{lem-lem11.4}.
\end{proof}
The localization sequence \cite[Proposition 3.2]{kok2023higher} is compatible with the pull-back of refined unramified cohomology in the following way.
\begin{lem}
    Let $f\colon X\to Y$ be a morphism of smooth algebraic schemes over $k$. Let $Z\subset Y$ be a smooth closed subset of codimension $c$ with open  complement $U$ with $\dim(U)=\dim(Y)$. Suppose that $Z'\coloneqq f^{-1}(Z)\subset X$ is smooth of codimension $c$, and its open complement $U'$ has dimension $\dim(X)$. Then we have the following commutative diagram of exact sequences
    \begin{equation*}
        \xymatrix{\ar[r]&H^{p-2c}_{q-c,\nr}(Z,n-c)\ar[r]\ar[d]&H^p_{q,\nr}(Y,n)\ar[r]\ar[d]&H^p_{q,\nr}(U,n)\ar[r]\ar[d]&H^{p+1-2c}_{q
+1-c,\nr}(Z,n-c)\ar[r]\ar[d]&\\
\ar[r]&H^{p-2c}_{q-c,\nr}(Z',n-c)\ar[r]&H^p_{q,\nr}(X,n)\ar[r]&H^p_{q,\nr}(U',n)\ar[r]&H^{p+1-2c}_{q+1-c,\nr}(Z',n-c)\ar[r]&,
    }
    \end{equation*}
    where all vertical arrows are induced by $f^*$.
\end{lem}
\begin{proof}
    The morphism $f$ can be factored as the composition of a regular embedding and a (flat) projection. Thus it suffices to consider that $f$ and $f|_{Z'}$ are  regular embeddings, as flat pull-back is naturally compatible with localization sequences.

    The commutativity of the first square and the second square is automatically induced by the compatibility of push-forward of closed immersion (\emph{resp.} open restriction) and flat pull-back (and residue map and cap product). We only discuss the third square here. Consider the following diagram
    \begin{equation*}
    \adjustbox{scale=.9}{
        \xymatrix{H^p_{q,\nr}(U,n)\ar[r]^{\hspace{-4mm}\pi^*}\ar[d]^{\delta}&H^p_{q,\nr}(U\times \G_m,n)\ar[r]^{\hspace{-4mm}\{t\}}\ar[d]^{\delta}&H^{p+1}_{q,\nr}(U\times \G_m,n+1)\ar[r]^{\hspace{4mm}-\delta}\ar[d]^{\delta}&H^{p+1}_{q,\nr}(N_{U'}U,n)\ar[d]^{\delta}\\
        H^{p+1-2c}_{q+1-c,\nr}(Z,n-c)\ar[r]^{\hspace{-4mm}\pi^*}&H^{p+1-2c}_{q+1-c,\nr}(Z\times \G_m,n-c)\ar[r]^{\hspace{-4mm}\{t\}}&H^{p+2-2c}_{q+1-c,\nr}(Z\times \G_m,n-c+1)\ar[r]^{\hspace{4mm}-\delta}&H^{p+1-2c}_{q+1-c,\nr}(N_{Z'}Z,n-c),
        }
        }
    \end{equation*}
    where the first square is commutative while the other two squares are anti-commutative by \cite[Definition 4.3(5)]{kerz2012cohomological}. Thus by the construction of pull-backs of refined unramified cohomology, the third square in this lemma commutes.
\end{proof}
\begin{prop}\label{prop-exact sequence is compatible with pullback}
Let $X\xrightarrow{f}Y$ be a morphism of smooth equi-dimensional schemes over a field $k$. Then the following diagram commutes
\begin{equation}\label{equ-exact sequence is compatible with pullback}
\xymatrix{H^{p+q-1}_{p-2,\nr}(Y,n)\ar[rr]^{\partial_Y}\ar[d]^{f^*}&&E^{p,q}_2(Y,n)\ar[rr]^{\iota_{Y,*}}\ar[d]^{f^{\bullet}}&&H^{p+q}_{p,\nr}(Y,n)\ar[d]^{f^*}\\
H^{p+q-1}_{p-2,\nr}(X,n)\ar[rr]^{\partial_X}&&E^{p,q}_2(X,n)\ar[rr]^{\iota_{X,*}}&&H^{p+q}_{p,\nr}(X,n),
}
\end{equation}
where $f^{\bullet}$ is the pull-back map on cycle modules defined in \cite[Section 12]{Ros96}.
\end{prop}
\begin{proof}
If $f$ is a flat morphism, one can easily get the commutativity of this diagram directly from the definition. By definition of $f^*$, it thus suffices to show the commutativity for $f$ a closed immersion. Let $\pi\colon Y\times \mathbb{G}_m\to Y$ and $p\colon N_XY\to X$ be projections.

Consider the following diagram
\begin{equation}
\adjustbox{scale=.9}{
\xymatrix{H^{p+q-1}_{p-2,\nr}(Y,n)\ar[rr]^{\partial}\ar[dd]^{\pi^*}&&E_2^{p,q}(Y,n)\ar[dd]^{\pi^*}\ar[rr]^{\iota_*}&&H^{p+q}_{p,\nr}(Y,n)\ar[dd]_{\pi^*}\\
    &(1)&&(5)&\\
    H^{p+q-1}_{p-2,\nr}(Y\times \mathbb{G}_m,n)\ar[rr]^{\partial}\ar[dd]^{\{t\}}&&E_2^{p,q}(Y\times \mathbb{G}_m,n)\ar[dd]^{\{t\}}\ar[rr]^{\iota_*}&&H^{p+q}_{p,\nr}(Y\times \mathbb{G}_m,n)\ar[dd]_{\{t\}}\\
    &(2)&&(6)&\\
    H^{p+q}_{p-2,\nr}(Y\times \mathbb{G}_m,n+1)\ar[rr]^{\partial}\ar[dd]^{-\delta^{p+q}_{p-2}}&&E_2^{p,q+1}(Y\times \mathbb{G}_m,n+1)\ar[dd]^{\partial}\ar[rr]^{\iota_*}&&H^{p+q+1}_{p,\nr}(Y\times \mathbb{G}_m,n+1)\ar[dd]_{-\delta^{p+q+1}_{p}}\\
    &(3)&&(7)&\\
    H^{p+q-1}_{p-2,\nr}(N_XY,n)\ar[rr]^{\partial}&&E_2^{p,q}(N_XY,n)\ar[rr]^{\iota_*}&&H^{p+q}_{p,\nr}(N_XY,n)\\
    &(4)&&(8)&\\
    H^{p+q-1}_{p-2,\nr}(X,n)\ar[rr]^{\partial}\ar[uu]_{p^*}&&E_2^{p,q}(X,n)\ar[uu]_{p^*}\ar[rr]^{\iota_*}&&H^{p+q}_{p,\nr}(X,n).\ar[uu]^{p^*}
}}
\end{equation}
As flat pull-backs of cohomology are compatible with localization sequences and purity isomorphisms, the squares $(1)$, $(4)$, $(5)$ and $(8)$ are commutative. Also, $(6)$ and $(7)$ commute as $\iota_*$ is compatible with $\{t\}$ and $\partial$. Here we remind the reader that the residue map $\partial$ in \cite{Ros96} is the negative of the residue map of cohomology with supports induced by the localization sequence.

The square $(2)$ is anti-commutative by \cite[Definition 1.1, R3e]{Ros96}. Moreover the square $(3)$ is also anti-commutative (see \cite[Lemma 2.4]{JS03}) and note that the residue map $\partial$ in \cite{Ros96} is the negative of the residue map of cohomology with supports from localization sequences. More specifically, let $\mathcal{X}\coloneqq D(Y,X)\setminus  W_{\rm sing}$ with $R\coloneqq (Y\times \mathbb{G}_m)\setminus U$ of codimension $p$ in $Y\times \mathbb{G}_m$ and $W$ is the closure of $R$ in $D(Y,X)$ (note that $R_{\rm sing}\subset W_{\rm sing}$, the codimension of $W_{\rm sing}$ is at least $p+1$ in $D(Y,X)$ and the codimension of $W_{\rm sing}\cap N_XY$ is at least $p+2$ in $N_X Y$). Let $\mathcal{V}\coloneqq \mathcal{X}\setminus W$, $\mathcal{U}\coloneqq \mathcal{X}\cap (Y\times \mathbb{G}_m)$, and let $\mathcal{Z}$ (resp. $\mathcal{Y}$) be the closed complement of $\mathcal{V}$ (resp. $\mathcal{U}$) in $\mathcal{X}$. By applying the proof of \cite[Lemma 2.4]{JS03} and noting that these schemes are smooth, we can use the purity isomorphism freely. Then one can get that the following diagram anti-commutes:
\begin{equation*}
\xymatrix{
H^{p+q}(U,n+1)\ar[rr]^{\partial\phantom{abdcfdd}}\ar[d]^{\partial}&& H^{q+1-p}(R\setminus R_{\rm sing},n+1-p)\ar[d]^{\partial}\\
    H^{p+q-1}({N_XY\setminus W},n)\ar[rr]^{\partial\phantom{abcdf}}&& H^{q-p}(W\setminus(W_{\rm sing}\cup R),n-p).
}
\end{equation*}
Taking a limit, one obtains the anti-commutativity of the square $(3)$. Therefore, by the definition of pullbacks, the diagram (\ref{equ-exact sequence is compatible with pullback}) is commutative.
\end{proof}
Now, we can extend the naturality of the long exact sequence in \cite[Theorem 1.1]{kok2023higher} as the following.
\begin{prop}\label{prop-Prop 2.15}
    Let $f\colon X\to Y$ be a morphism between smooth quasi-projective equi-dimensional schemes defined over a field $k$, and $M\coloneqq \Z/m\Z$, where $m$ is an integer invertible in $k$. Then there is a commutative diagram of exact sequences
    \begin{small}
     \begin{equation*}
        \xymatrix{\ar[r]&\CH^b(Y,n;M)\ar[r]^{\mathrm{cl}}\ar[d]^{f^*}&H^{2b-n}(Y,M(b))\ar[r]^{\rm restr.\phantom{abc}}\ar[d]^{f^*}& H^{2b-n}_{b-n-1,\nr}(Y,M(b))\ar[r]^{\theta}\ar[d]^{f^*}&\CH^b(Y,n-1;M)\ar[r]\ar[d]^{f^*}&\\
        \ar[r]&\CH^b(X,n;M)\ar[r]^{\mathrm{cl}}&H^{2b-n}(X,M(b))\ar[r]^{\rm restr.\phantom{abc}}& H^{2b-n}_{b-n-1,\nr}(X,M(b))\ar[r]^{\theta}&\CH^b(X,n-1;M)\ar[r]&.
        }
    \end{equation*}  
    \end{small}
\end{prop}
\begin{proof}
The naturality of commutativity between the first and second squares is natural, where the first is due to the functoriality of higher cycle class maps and the second is due to \Cref{prop-exact sequence is compatible with pullback}. So here, we only need to consider the commutativity of the third square. Since the naturality of the third square holds for flat maps as well, we can assume that $f$ is an embedding. 

First, let us review the construction of the connecting map 
\[\theta:H^{2b-n}_{b-n-1,\nr}(*,M(b))\to \CH^b(*,n-1;M).\] According to \cite[Equation (27)]{kok2023higher}, the connecting map $\theta$ is given by  the following commutative diagram:
\begin{equation}
\xymatrix{H^{2b-n}(F_{b-n}*,M(b))\ar[rrr]^{\partial\phantom{abcdeeee}}\ar[d]_{\theta}&&&\varinjlim\limits_{\codim Z=b-n+1}H^{n-1}(Z,M(n-1))\\
 \CH^b(*,n-1;M)&&&\varinjlim\limits_{\codim Z=b-n+1}\CH^{n-1}(Z,n-1;M)\ar[u]^{\sim}_{\mathrm{cl}}\ar[lll]_{\iota_*\phantom{ddddddddd}},
       }
\end{equation}
where $Z$ runs through all closed subset of `$*$' with $\codim(Z)=b-n+1$ and the left vertical arrow is an isomorphism by \cite[Proposition 4.9]{kok2023higher}.

Recalling that the third square commutes for flat morphisms, the naturality of the third square for smooth embeddings is due to the following commutative diagram (here we omit the coefficients $M$)
\begin{equation*}
\adjustbox{width=\textwidth}{
    \xymatrix{H^{2b-n}(F_{b-n}(Y\times \G_m),b)\ar[r]^{\partial\phantom{dd}}\ar[d]^{\{t\}}&\varinjlim\limits_{\codim Z=b-n+1}H^{n-1}(Z,n-1)\ar[d]^{\{t\}}& \varinjlim\limits_{\codim Z=b-n+1}\CH^{n-1}(Z,n-1)\ar[r]^{\hspace{8mm}\iota_*}\ar[l]_{\hspace{-4mm}\sim}\ar[d]^{\{t\}}& \CH^b(Y\times \G_m,n-1)\ar[d]^{\{t\}}\\
    H^{2b-n+1}(F_{b-n}(Y\times \G_m),b+1)\ar[r]^{\phantom{ddd}\partial}\ar[d]^{-\delta}&\varinjlim\limits_{\codim Z''=b-n+1}H^{n}(Z',n)\ar[d]^{-\delta}& \varinjlim\limits_{\codim Z'=b-n+1}\CH^{n}(Z',n)\ar[r]^{\hspace{8mm}\iota_*\phantom{dd}}\ar[l]_{\hspace{-4mm}\sim}\ar[d]^{-\delta}& \CH^{b+1}(Y\times \G_m,n)\ar[d]^{-\delta}\\
    H^{2b-n}(F_{b-n}(N_XY),b)\ar[r]^{\partial\phantom{ddddd}}&\varinjlim\limits_{\codim Z''=b-n+1}H^{n-1}(Z'',n-1)& \varinjlim\limits_{\codim Z''=b-n+1}\CH^{n-1}(Z'',n-1)\ar[r]^{\phantom{dd}\hspace{8mm}\iota_*}\ar[l]_{\phantom{dd}\hspace{-4mm}\sim}& \CH^b(N_XY,n-1),
      }}
\end{equation*}
where $Z$ (resp. $Z'$ and $Z''$) runs through all closed subsets of $Y\times \G_m$ (resp. $Y\times \G_m$ and $N_XY$) of codimension $b-n+1$. Note that here we implicitly used the étale Borel--Moore homology as in \cite{kok2023higher}, as the $Z$, $Z'$ and $Z''$ are not necessarily smooth anymore; the construction in \Cref{section general pullbacks} still works in this case.
\end{proof}
Finally, we write down the projection formula for refined unramified cohomology of smooth schemes.
\begin{lem}[Projection formula]\label{lem-projection formual for refined}
  Consider a fiber  product
  \begin{equation*}
      \xymatrix{Y'\ar[r]^{f'}\ar[d]^{g'}&X'\ar[d]^{g}\\
      Y\ar[r]^{f}&X,
      }
  \end{equation*}
  where $g$ is smooth and proper, and $X$ and $Y$ are smooth over $k$. Then
  \begin{equation*}
      f^*\circ g_*=g'_*\circ (f')^*.
  \end{equation*}
\end{lem}
\begin{proof}
    See the proof of \cite[Proposition 12.5]{Ros96}.
\end{proof}

\section{Action of Cycles}
In this section, we redefine a correspondence action on the refined unramified cohomology groups, inspired by Fulton \cite{fulton2013intersection}. We also show that the correspondence action on refined unramified cohomology is compatible with the correspondence action on Bloch--Ogus spectral sequences for smooth proper schemes via the long exact sequence (\ref{equ-the canonical long exact sequence}).

Let us first review the construction of the correspondence action on Bloch--Ogus type spectral sequences due to Rost \cite{Ros96}, which is the same as the construction in \cite[Section 9]{CTV12}. Let $M$ be the \emph{cycle module} in \cite[Remark 1.11]{Ros96} or $M=K_*$ the \emph{Milnor $K$-rings}. If $X$ is equi-dimensional of
dimension $d_X$, we let $A^p(X,M)\coloneqq A_{d_X-p}(X,M)$, where $A_*(X,M)$ represents the homology groups of the cycle modules $C_*(X,M)$, $C_p(X,M)\coloneqq \coprod_{x\in X_{(p)}}M(x)$ and $X_{(p)}$ is the set of $p$-dimensional points of $X$ (\emph{ref.} \cite[Section 5]{Ros96}).
\begin{define}
Let $X$, $Y$ and $Z$ be algebraic equi-dimensional schemes over a field $k$, and suppose $Y$ is smooth and proper with dimension $d_Y$. We define the correspondence action
\begin{align*}
\CH^c(X\times Y)\otimes A^p(Y\times Z,M)&\to A^{c+p-d_Y}(X\times Z,M)\\
[\Gamma]\otimes [\alpha]&\mapsto[\Gamma]_*[\alpha]
\end{align*}
as the composition
\begin{align}\label{equ-corr actions 1}
\begin{split}
   \CH^c(X\times Y)\otimes A^p(Y\times Z,M)\xrightarrow{\times}A^{c+p}(X\times Y\times Y\times Z,M)\\
    \xrightarrow{\Delta^{\bullet}}A^{c+p}(X\times Y\times Z)\xrightarrow{p_*}A^{c+p-d_Y}(X\times Z,M),  
\end{split}  
\end{align}
where $\times$ is the cross product introduced in \cite[Section 14]{Ros96}, $\Delta^{\bullet}$ is the pull-back of cycle modules along the diagonal $X\times Y\times Z\hookrightarrow X\times Y\times Y\times Z$ from \cite[Section 12]{Ros96} and $p_*$ is the push-forward of cycle modules along the proper morphism $p\colon X\times Y\times Z\to X\times Z $ from \cite[Section 3]{Ros96}. 
\end{define}

\begin{rem}\label{remark compatibility}
If we moreover assume that $X$ and $Z$ are smooth, then using the compatibilities in \cite[(14.5)]{Ros96}, we see that this definition is the same as the composition 
\begin{align*}
\begin{split}
  \CH^c(X\times Y)\otimes A^p(Y\times Z,M)\xrightarrow{p_{XY}^*\otimes p_{YZ}^*}\CH^c(X\times Y\times Z)\otimes A^p(X\times Y\times Z,M) \xrightarrow{\times}\\
  A^{c+p}(X\times Y\times Z\times X\times Y\times Z,M)\xrightarrow{\Delta^{\bullet}}A^{c+p}(X\times Y\times Z)\xrightarrow{p_{XZ}*}A^{c+p-d_Y}(X\times Z,M).   
\end{split}
\end{align*}
\end{rem}
One can check that this correspondence action is compatible with the composition of correspondences by the work of Rost \cite{Ros96}. Moreover, this action induces the correspondence action on the second page $E_2^{p,q}$ of the Bloch--Ogus/local-to-global spectral sequence. The action is also compatible with the correspondence action on cohomology groups via the Bloch--Ogus spectral sequence by construction.

\subsection{Cross Product}\label{subsection Cross Product}
Here we define a \emph{cross-product action} of the Chow groups on the refined unramified cohomology groups.
\begin{define}[External/Cross Product I]\label{def-external/cross product I}
 Let $X$ and $Y$ be smooth equi-dimensional algebraic schemes over $k$. We have a well-defined pairing $$\CH^c(X)\otimes H^i_{j,\nr}(Y,n)\xrightarrow{\times}H^{i+2c}_{j+c,\nr}(X\times Y,n+c).$$   
\end{define}
\begin{proof}
For any $[\alpha]\in H^i_{j,\nr}(Y,n)$, let us take an open subset $U\subset Y$ with $\codim_Y(Y\setminus U)=j+2$ such that $\alpha\in H^i(U,n)$. For any $[\Gamma]\in \CH^c(X)$, taking any representative element $\Gamma$ with $|\Gamma|\coloneqq {\rm Supp}(\Gamma)$, then $\Gamma\times Y$ is a representative of $q^*([\Gamma])\in \CH^c(X\times Y)$ where $p\colon X\times Y\to Y$ and $q\colon X\times Y\to X$. Now we define $[\Gamma]\times [\alpha]$ as the class of the image of $\alpha$ under the following composition of morphisms (we denote as $\Gamma\times \alpha$), inspired by \cite[Section 4]{schreieder2022moving},
\begin{equation*}
    \xymatrix{H^i(U,n)\ar[d]^{p^*}\\
    H^i(X\times U,n)\ar[d]^{\cl^{X\times U}_{|\Gamma|\times Y}(q^*(\Gamma)|_{X\times U})\cup}\\
    H^{i+2c}_{|\Gamma|\times Y}(X\times U,n+c)\ar[d]^{\sim}_{\rm exc}\\
    H^{i+2c}_{|\Gamma|\times Y}((X\times Y)\setminus(|\Gamma|\times Z),n+c)\ar[d]^{\iota_*}\\
    H^{i+2c}((X\times Y)\setminus(|\Gamma|\times Z),n+c)\nospacepunct{,}
}
\end{equation*}
where $Z\coloneqq Y\setminus U$. Note that $\codim_{X\times Y}(|\Gamma|\times Z)=j+2+c$, so $\Gamma\times \alpha$ induces an element in $H^{i+2c}_{j+c,\nr}(X\times Y,n+c)$. Moreover, from this construction, one can see that cross product is additive for $\CH^c(X)$ and $H^i_{j,\nr}(Y,n)$, and it is independent of the choice of $\alpha$. After all, if we take $[\alpha_1]=[\alpha_2]=[\alpha]$ such that $\alpha_j\in H^i(U_j,n)$ for $j=1,2$, one can take a smaller $U\subset U_1\cap U_2$ such that $\alpha_1=\alpha_2$ on $U$ and then show that they have the same image in $H^{i+2c}((X\times Y)\setminus(|\Gamma|\times Z),n+c)$ with $Z\coloneqq Y\setminus U$.

Finally, let us show that the cross product is independent of the choice of $\Gamma$. Suppose that $[\Gamma_1]=[\Gamma_2]=[\Gamma]\in \CH^c(X)$. Then we can find a $(c-1)$-codimensional closed subset $W\subset X$ such that $\Gamma_1-\Gamma_2$ is rationally equivalent to zero on $W$. Consider the following commutative diagram
\begin{equation*}
\adjustbox{width=\textwidth}{
\xymatrix{H^i(U,n)\ar[d]^{p^*}\ar@{=}[r]&H^i(U,n)\ar[d]^{p^*}\ar@{=}[r]&H^i(U,n)\ar[d]^{p^*}\\
    H^i(X\times U,n)\ar[d]^{\cl^{X\times U}_{|\Gamma_1|\times Y}(q^*(\Gamma_1)|_{X\times U})\cup}\ar@{=}[r]&H^i(X\times U,n)\ar[d]^{\cl^{X\times U}_{W\times Y}(q^*(\Gamma_2)|_{X\times U})\cup}_{{\cl^{X\times U}_{W\times Y}(q^*(\Gamma_1)|_{X\times U})}\cup}\ar@{=}[r]&H^i(X\times U,n)\ar[d]^{\cl^{X\times U}_{|\Gamma_2|\times Y}(q^*(\Gamma_2)|_{X\times U})\cup}\\
    H^{i+2c}_{|\Gamma_1|\times Y}(X\times U,n+c)\ar[d]^{\sim}_{\rm exc}\ar[r]^{\iota_*}&H^{i+2c}_{W\times Y}(X\times U,n+c)\ar[d]^{\sim}_{\rm exc}&H^{i+2c}_{|\Gamma_2|\times Y}(X\times U,n+c)\ar[d]^{\sim}_{\rm exc}\ar[l]_{\iota_*}\\
    H^{i+2c}_{|\Gamma_1|\times Y}((X\times Y)\setminus(|\Gamma_1|\times Z),n+c)\ar[d]^{\iota_*}\ar[r]^{\rm \iota_*\circ restr.}&H^{i+2c}_{W\times Y}((X\times Y)\setminus(W\times Z),n+c)\ar[d]^{\iota_*}&H^{i+2c}_{|\Gamma_2|\times Y}((X\times Y)\setminus(|\Gamma_2|\times Z),n+c)\ar[d]^{\iota_*}\ar[l]_{\rm \iota_*\circ restr.}\\
    H^{i+2c}((X\times Y)\setminus(|\Gamma_1|\times Z),n+c)\ar[r]^{\rm restr.}&H^{i+2c}((X\times Y)\setminus(W\times Z),n+c)&H^{i+2c}((X\times Y)\setminus(|\Gamma_2|\times Z),n+c)\ar[l]_{\rm restr.}.
}
}
\end{equation*}
Note that $\codim_{X\times Y}(W\times Z)=j+2+c-1=j+c+1$, so $\Gamma_1\times \alpha$ and $\Gamma_2\times \alpha$ have the same image in $H^{i+2c}((X\times Y)\setminus(W\times Z),n+c)$ so also in $H^{i+2c}_{j+c,\nr}(X\times Y,n+c)$.
\end{proof}

\begin{define}[External/Cross Product II]\label{def-external/cross product II}
Let $X$ and $Y$ be smooth equi-dimensional algebraic schemes over $k$. We have another pairing 
\[
\CH^c(X)\otimes H^i_{j,\nr}(Y,n)\xrightarrow{\times'}H^{i+2c}_{j+c,\nr}(X\times Y,n+c).
\]
\end{define}
\begin{proof}
 We give the construction here and show the well-definedness later in \Cref{cross-2-welldef}. For any $[\alpha]\in H^i_{j,\nr}(Y,n)$, let us take an open subset $U\subset Y$ with $\codim_Y(Y\setminus U)=j+2$ such that $\alpha\in H^i(U,n)$. For any $[\Gamma]\in \CH^c(X)$ such that the representative element $\Gamma$ is an effective prime cycle with $|\Gamma|\coloneqq  {\rm Supp}(\Gamma)$, we let $\pi\colon  |\Gamma|\times U\to U$ be the projection. Now we define $[\Gamma]\times' [\alpha]$ (also denoted as $\Gamma\times' \alpha$) as the image of $\alpha$ under the composition defined in \ref{diagram-def-external/cross product II}, where $Z\coloneqq  Y\setminus U$. For general $[\Gamma]\in \CH^c(X)$ with representative element $\Gamma\coloneqq  \sum\limits_i n_i\Gamma_i$, where $\Gamma_i$ is effective prime, we define 
 \[
 [\Gamma]\times' [\alpha]\coloneqq  \sum\limits_i n_i ([\Gamma_i]\times' [\alpha])\in H^{i+2c}_{j+c,\nr}(X\times Y,n+c).\qedhere
 \]
\end{proof}
\begin{equation*}\label{diagram-def-external/cross product II}
\xymatrix{H^i(U,n)\ar[d]^{\pi^*}\\
H^i(|\Gamma|\times U,n)\ar[d]^{\rm Gysin}\\
 H^{i+2c}_{|\Gamma|\times U}(X\times U,n+c)\ar[d]^{\sim}_{\rm exc}\\
 H^{i+2c}_{|\Gamma|\times U}((X\times Y)\setminus (|\Gamma|\times Z),n+c)\ar[d]^{\iota_*}\\
 H^{i+2c}((X\times Y)\setminus (|\Gamma|\times Z),n+c).
    }\tag{Morphism \ref{def-external/cross product II}}
\end{equation*}
To show that the definition of $\times'$ is well-defined, we first recall the projection formula. A proof using motives is given in \cite[Proposition 2.2.2]{levine1998mixed}. For the convenience of the reader, we also give a more elementary proof here. 
\begin{lem}[Projection Formula]\label{lem-Gysin homomorphism}
    Let X be smooth of pure dimension $d$ over $k$ and let $i\colon Z\hookrightarrow  Y\hookrightarrow X$ be closed embeddings, where $Y$ is of pure codimension $c$ in $X$. There exists a Gysin homomorphism $P_{X,Y}\colon H^a_Z(Y,n)\to H^{a+2c}_Z(X,n+c)$ as in \cite[Definition 4.3(6)(7)]{kerz2012cohomological}. In particular, if $Z=Y$ is integral and we write $i_*\coloneqq  P_{X,Z}$, then $i_*(1)=\cl^X_Z(Z)$ for $a=n=0$ and $i_*\alpha\cup \beta=i_*(\alpha\cup i^*\beta)$ for any $\alpha\in H^a(Z,n)$ and any $\beta \in H^b(X,m)$.
\end{lem}
\begin{proof}
We write $A(n)\coloneqq  \Z/\ell^r\Z(n)$ for any prime $\ell\neq {\rm char}(k)$.
All functors in this proof are derived functors in the derived category of abelian \'etale sheaves. Let $\pi_Z$(\emph{resp.} $\pi_X$, $\pi_Y$) be the structure map of $Z$ (\emph{resp.} $X$, $Y$) over $k$.

\emph{Step 1: Construction of Gysin homomorphism $P_{X,Y}$.}

Write $i_Y\colon Y\hookrightarrow X$ and $i_Z\colon Z\hookrightarrow Y$ for the closed immersions, and consider the morphism in the derived category of $Y$
\begin{equation}\label{equ-construction of Gysin homo}
    \pi^*_YA(n)\to \pi^!_YA(n-d+c)[2c-2d]=i_Y^!\pi_X^!A(n-d+c)[2c-2d]\cong i_Y^!\pi_X^*A(c+n)[2c],
\end{equation}
where the first arrow is induced by the fundamental class of $Y$ and the last isomorphism is induced by the canonical isomorphism $\pi^*_X(d)[2d]\xrightarrow{\cong}\pi_X^!$ as $X$ is smooth of pure dimension $d$. Applying functor $R\Gamma(Z,i_Z^!-)$ to the morphism (\ref{equ-construction of Gysin homo}) and noting that there is an isomorphism $R\Gamma_Z(X,-)\cong R\Gamma(Z,i^!(-))$, we get the morphism $P_{X,Y}\colon H^a_Z(Y,n)\to H^{a+2c}_Z(X,n+c)$. By the definition of $P_{X,Y}$, if $Y$ is smooth, the first arrow in (\ref{equ-construction of Gysin homo}) is also an isomorphism, which is exactly the construction of purity isomorphism (or Gysin homomorphism). Moreover, if $Z=Y$, the image of $i_*(1)$ for $a=n=0$ is the cycle class $\cl(Z)\coloneqq  \cl_Z^X(Z)$ (\emph{ref.} \cite[Section 1]{fujiwara2002proof}). Indeed if $Z=Y$, this constructs the morphism $H^{BM}_{2d_Z}(Z,d_Z)\to H^{BM}_{2d_Z}(X,d_Z)$, which maps the generator of $H^{BM}_{2d_Z}(Z,d_Z)\cong H^0(Z_{sm},0)$ to $\cl^X_Z(Z)$ (in $H^{2(d_X-d_Z)}(X)$ if $X$ is smooth). One can also get this statement as follows: writing $j\colon Z_{\rm sm}\hookrightarrow Z$ and applying $\id\to j_*j^*$ followed with $R\Gamma(Z,-)$ to (\ref{equ-construction of Gysin homo}), we can get a commutative diagram
\begin{equation*}
    \xymatrix{H^0(Z,0)\ar[r]^{i_*}\ar[d]^{\sim}&H^{2c}_{Z}(X,c)\ar[d]^{\sim}\\
    H^0(Z_{\rm sm},0)\ar[r]^{\hspace{-7mm}\sim}&H^{2c}_{Z_{\rm sm}}(X\setminus Z_{\rm sing},c),
     }
\end{equation*}
where the left isomorphism is induced by semi-purity, hence we get $i_*(1)=\cl_Z^X(Z)$ by the usual definition of cycle class via $H^0(Z_{\rm sm},0)\to H^{2c}_{Z_{\rm sm}}(X\setminus Z_{\rm sing},c)$. In particular, $i_*\colon H^0(Z,0)\to H^{2c}_Z(X,c)$ is an isomorphism.

\emph{Step 2: Compatibility with localization sequences.} 

Now if $i_Z\colon Z\xrightarrow{\Bar{i}}Z'\xrightarrow{i_{Z'}}Y$ are closed immersions, with $j\colon Z'\setminus Z\to Z'$ and $j_Y\colon Y\setminus Z\to Y$ open complements, and with $i'\colon Z'\setminus Z\to Y\setminus Z$ closed, then the localization sequence is defined as follows. We start with the distinguished triangle
\begin{equation}\label{equ-distinguished triangles}
\to \Bar{i}_\ast \Bar{i}^!\to\id\to j_\ast j^!\xrightarrow{+1},    
\end{equation}
then plugging in $i_{Z'}^!\pi_Y^\ast A$ and applying $R(\pi_{Z'})_\ast$ ($=R\Gamma(Z',-)$) gives the distinguished triangle 
\[
\to (\pi_Z)_\ast i_Z^!\pi_Y^\ast A\to (\pi_{Z'})_\ast i_{Z'}^!\pi_Y^\ast A\to(\pi_{Z'\setminus Z})_\ast(i')^!\pi_{Y\setminus Z}^\ast A\xrightarrow{+1}.
\]
Hence taking cohomology gives the long exact sequence
\[
\cdots\to H_Z^a(Y,n)\to H_{Z'}^a(Y,n)\to H_{Z'\setminus Z}^a(Y\setminus Z,n)\to H^{a+1}_{Z}(Y,n)\to\cdots.
\]
So we get immediately a morphism of distinguished triangles by plugging in the morphism $i^!_{Z'}\pi_Y^\ast A(n)\to i_{Z'/X}^!\pi_X^\ast A(n+c)[2c]$ from above, where $i_{Z'/X}\colon Z'\hookrightarrow X$, into the distinguished triangle (\ref{equ-distinguished triangles}). And after taking cohomology, we thus get a commutative diagram of desired long exact sequences.

\emph{Step 3: Projection formula.} 

We write $A_X\coloneqq  \pi_X^\ast A$ for the sheaf $A$ on $X$ (to make notation less heavy, we stop writing the Tate twists here). The cup product (say on $X$) is defined via the composition
\[
(\pi_X)_\ast A_X[a]\otimes(\pi_X)_\ast A_X[b]\to(\pi_X)_\ast(A_X[a]\otimes A_X[b])\overset{\sim}{\to}(\pi_X)_\ast A_X[a+b],
\]
where the first map is the composition of the projection formula $$(\pi_X)_\ast A_X[a]\otimes (\pi_X)_\ast A_X[b]\overset{\sim}{\to}(\pi_X)_\ast(A_X[a]\otimes\pi_X^\ast(\pi_X)_\ast A_X[b])$$ composed with the natural map $\pi_X^\ast(\pi_X)_\ast\to\id$. 
 Let $i\colon Z\to X$ be the inclusion of a closed of codimension $c$ from before. We have a commutative diagram
\[
\adjustbox{width=\textwidth}{
\begin{tikzcd}
(\pi_Z)_* A_Z[a]\otimes(\pi_X)_* A_X[b]\ar[r,"\id\otimes i^\ast"]\ar[d,"i_\ast\otimes\id"]&(\pi_Z)_* A_Z[a]\otimes(\pi_Z)_* A_Z[b]\ar[r]\ar[d,"i_\ast\otimes\id"]&(\pi_Z)_*(A_Z[a]\otimes A_Z[b])\ar[d,"(\pi_Z)_*(i_\ast\otimes\id)"]\\
(\pi_Z)_* i^! A_X[a+2c]\otimes(\pi_X)_* A_X[b]\ar[r,"\id\otimes i^\ast"]&(\pi_Z)_* i^! A_X[a+2c]\otimes(\pi_Z)_* A_Z[b]\ar[r]&(\pi_Z)_*(i^!A_X[a+2c]\otimes A_Z[b]).
\end{tikzcd}
}
\]
Now if we compose the top horizontal map with $$(\pi_Z)_*(A_Z[a]\otimes A_Z[b])\overset{\sim}{\to}(\pi_Z)_* A_Z[a+b]\overset{i_\ast}{\to}(\pi_Z)_*i^!A_X[a+b+2c],$$ we obtain $i_\ast(\alpha\cup i^\ast\beta)$. We also note that $i_\ast\alpha\cup\beta$ is defined by the composition
\begin{align*}
(\pi_Z)_* A_Z[a]\otimes(\pi_X)_* A_X[b]&\xrightarrow{i_\ast\otimes\id} (\pi_Z)_* i^! A_X[a+2c]\otimes(\pi_X)_* A_X[b]\to\pi_X(i_\ast i^! A_X[a+2c]\otimes A_X[b])\\
&\xrightarrow{\sim} (\pi_X)_* (i_\ast i^!A_X[a+2c+b])=(\pi_Z)_* i^! A_X[a+b+2c].
\end{align*}
Using the first commutative diagram, these two maps are related by the following commutative diagrams. First of all, we can extend the first diagram on the right by the square
\[
\begin{tikzcd}
(\pi_Z)_\ast(A_Z[a]\otimes A_Z[b])\ar[r,"\cong"]\ar[d,"(\pi_Z)_\ast(i_\ast\otimes\id)"]&(\pi_Z)_\ast A_Z[a+b]\ar[d,"i_\ast"]\\
(\pi_Z)_\ast(i^!A_X[a+2c]\otimes A_Z[b])\ar[r,"\cong"]&(\pi_Z)_\ast i^!A_X[a+b+2c]\nospacepunct{.}
\end{tikzcd}
\]
Connecting it to $i_\ast\alpha\cup\beta$ via the following diagram
\[
\begin{tikzcd}
(\pi_Z)_\ast i^!A_X[a+2c]\otimes(\pi_X)_* A_X[b]\ar[r,"\id\otimes i^\ast"]\ar[d]&(\pi_Z)_\ast i^!A_X[a+2c]\otimes (\pi_Z)_\ast A_Z[b]\ar[d]\\
(\pi_X)_\ast(i_\ast i^! A_X[a+2c]\otimes A_X[b])\ar[r,"\text{pr.~form.}"]\ar[d,"\cong"]&(\pi_Z)_\ast(i^!A_X[a+2c]\otimes A_Z[b])\ar[d,"\cong"]\\
(\pi_X)_\ast(i_\ast i^!A_X[a+2c+b])\ar[r,equal]&(\pi_Z)_\ast i^!A_X[a+b+2c]\nospacepunct{.}
\end{tikzcd}
\]
The crucial part is showing the commutativity of the top square of the last diagram, so let us give some details. For convenience, we write $A=i^!A_X[a+2c]$ and $B=A_X[b]$ (in fact, we show commutativity for arbitrary $A$ and $B$). Consider the following diagram
\[
\begin{tikzcd}
(\pi_X)_\ast i_\ast A\otimes(\pi_X)_\ast B\ar[r,"\id\otimes i^\ast"]\ar[d]&(\pi_X)_\ast i_\ast A\otimes(\pi_X)_\ast i_\ast i^\ast B\ar[d]&~\\
(\pi_X)_\ast(i_\ast A\otimes B)\ar[r,"\id\otimes i^\ast"]\ar[d,"\cong"]\ar[d,swap,"\text{pr.~form.}"]&(\pi_X)_\ast(i_\ast A\otimes i_\ast i^\ast B)\ar[d,swap,"\text{pr.~form.}"]\ar[d,"\cong"]&~\\
(\pi_X)_\ast i_\ast(A\otimes i^\ast B)\ar[r,"\id\otimes i^\ast"]&(\pi_X)_\ast i_\ast(A\otimes i^\ast i_\ast i^\ast B)\ar[r,"i^\ast i_\ast\to\id"]&(\pi_X)_\ast i_\ast (A\otimes i^\ast B)\nospacepunct{.}
\end{tikzcd}
\]
This diagram clearly commutes. In fact, it decomposes the first square of the previous diagram as follows. Note that the composition of the right vertical maps with $i^\ast i_\ast\to\id$ equals the map $(\pi_X)_\ast i_\ast A\otimes (\pi_X)_\ast i_\ast i^\ast B=(\pi_Z)_\ast A\otimes(\pi_Z)_\ast i^\ast B\to(\pi_Z)_\ast(A\otimes i^\ast B)$, so the right vertical map of the first square in the diagram before. The left two vertical maps are the other two. Now we conclude by noting that the composition of the bottom horizontal map is the identity as we have an adjunction $i^\ast\vdash i_\ast$, so that the composition $i^\ast\xrightarrow{\id\to i_\ast i^\ast}i^\ast i_\ast i^\ast\xrightarrow{i^\ast i_\ast\to\id}i^\ast $ is the identity.
\end{proof}

\begin{cor}\label{cross-2-welldef}
\Cref{def-external/cross product II} is well-defined and bi-additive, and coincides with the one in \Cref{def-external/cross product I}.
\end{cor}
\begin{proof}
It is easy to see that this pairing is bi-additive. Note that the \Cref{def-external/cross product I} is well-defined, so it suffices to show that these two definitions coincide. This is given by the following commutative diagram (here $\Gamma$ is prime)
\begin{equation*}
\xymatrix{H^i(U,n)\ar@{=}[r]\ar[d]^{p^*}&H^i(U,n)\ar[d]^{\pi^*}\\
H^i(X\times U,n)\ar[r]^{i^*}\ar[d]^{\cl^{X\times U}_{|\Gamma|\times Y}(q^*(\Gamma)|_{X\times U})\cup}&H^i(|\Gamma|\times U,n)\ar[d]^{i_*}\\
H^{i+2c}_{|\Gamma|\times U}(X\times U,n+c)\ar@{=}[r]&H^{i+2c}_{|\Gamma|\times U}(X\times U,n+c),
}
\end{equation*}
where $i\colon |\Gamma|\times U\hookrightarrow X\times U$ is a closed immersion and the last square commutes by \Cref{lem-Gysin homomorphism}.
\end{proof}

\begin{prop}\label{prop-composition-cross}
Let $X$, $Y$ and $Z$ be smooth proper schemes and $\alpha\in H^i_{j,\nr}(Z,n)$, $\Gamma\in \CH^c(Y)$ and $\Gamma'\in \CH^{c'}(X)$. Then $(\Gamma'\times\Gamma)\times\alpha=\Gamma'\times(\Gamma\times\alpha)\in H^{i+2(c+c')}_{j+c+c',\nr}(X\times Y\times Z,n+c+c')$.
\end{prop}
\begin{proof}
Suppose $\alpha$ is represented by some element in $H^i(U,n)$ with $F_{j+1}Z\subset U$. Writing out the definition of \Cref{def-external/cross product I}, this proposition essentially follows from the commutativity of the following diagram
\[
\begin{tikzcd}
H^i(U,n)\ar[rr, "q_U^\ast"]\ar[d,"p_U^\ast"]&~&H^p_{q,\nr}(X\times Y\times U,n)\ar[rr,"\cl^{X\times Y\times U}_{\Gamma'\times\Gamma'\times U}(\Gamma'\times\Gamma\times U)\cup"]&~&H^{i+2(c+c')}_{\Gamma'\times\Gamma\times U}(X\times Y\times U,n+c+c')\\
H^i(Y\times U,n)\ar[rr,"\cl^{Y\times U}_{\Gamma\times U}(p_Y^\ast\Gamma)\cup"]&~&H^{i+2c}_{\Gamma\times U}(Y\times U,n+c)\ar[rr,"q^\ast"]&~&H^{i+2c}_{X\times\Gamma\times U}(X\times Y\times U,n+c)\ar[u,"\cl^{X\times Y\times U}_{\Gamma'\times Y\times U}(q_X^\ast\Gamma')\cup"]\nospacepunct{,}
\end{tikzcd}
\]
where the maps are defined by 
\[
\begin{tikzcd}
X&X\times Y\times U\ar[l,"q_X"]\ar[d,"q"]\ar[dl,"q_Y"]\ar[dr,"q_U"]&\\
Y&Y\times U\ar[l,"p_Y"]\ar[r,"p_U"]&U\nospacepunct{.}
\end{tikzcd}
\]
The commutativity of the above diagram then follows from compatibility properties of the cup-product given in \cite[(C4e) and (C4f)]{schreieder2022moving}.
\end{proof}

\begin{prop}\label{prop-exact sequence is compatible with cross product}
Let $f\colon X\to Y$ be a morphism of smooth equi-dimensional schemes over a field $k$. For any $[\Gamma]\in \CH^c(X)$, the following diagram commutes
\begin{equation}\label{equ-exact sequence is compatible with cross product}
 \xymatrix{H^{p+q-1}_{p-2,\nr}(Y,n)\ar[rr]^{\partial_Y}\ar[d]^{[\Gamma]\times}&&E^{p,q}_2(Y,n)\ar[rr]^{\iota_{Y,*}}\ar[d]^{[\Gamma]\times}&&H^{p+q}_{p,\nr}(Y,n)\ar[d]^{[\Gamma]\times}\\
H^{p+q-1+2c}_{p-2+c,\nr}(X\times Y,n+c)\ar[rr]^{\partial_{X\times Y}}&&E^{p+c,q+c}_2(X\times Y,n+c)\ar[rr]^{\iota_{X\times Y,*}}&&H^{p+q+2c}_{p+c,\nr}(X\times Y,n+c).
}   
\end{equation}
\end{prop}
\begin{proof}
Suppose that $\Gamma$ is an effective prime cycle with generic point $\gamma$. For any $[\alpha]\in H^{p+q-1}_{p-2,\nr}(Y,n)$, there exists an open subset $U\subset Y$ with complement $Z$ such that $Z$ has pure codimension $p$ in $Y$ and $\alpha\in H^{p+q-1}(U,n)$. 

First, let us show the commutativity of the left square. Taking $V\coloneqq Y\setminus Z_{\rm sing}$, then $F_pY\subset V$. And we have the commutative \ref{diagram-exact sequence is compatible with cross product1}, where the morphisms $r$ are the restriction morphisms for open subsets. On the other hand, let $z_i(1\leq i\leq r)$ be the generic/maximal points of $X\times Z$, $m_i\coloneqq {\rm length}(\mathcal{O}_{X\times Z,z_i})$ and let $p_i$ be the induced morphism $\overline{\{z_i\}}\to Z$ by the morphism $p$. Then we get another commutative \ref{diagram-exact sequence is compatible with cross product2}, where $i\colon  \Gamma_{\rm sm}\times Z_{\rm sm}\hookrightarrow (X\setminus \Gamma_{\rm sing})\times Z_{\rm sm}$, and here the commutativity of the last square is given by projection formula for smooth schemes. Combining these two diagrams, after applying the purity isomorphism and recalling the definition of $[\Gamma]\times\colon  E_2^{p,q}(Y,n)\to E_2^{p+c,q+c}(X\times Y,n+c)$ in \cite[Section 14]{Ros96}, we get that the left square commutes.
   \begin{equation*} \label{diagram-exact sequence is compatible with cross product1}
   \adjustbox{width=\textwidth}{
    \xymatrix{H^{p+q-1}(U,n)\ar[r]^{\partial}\ar[d]^{p^*}&H^{p+q}_Z(Y,n)\ar[d]^{p^*}\ar[r]^{r}&H^{p+q}_{Z_{\rm sm}}(V,n)\ar[d]^{p^*}\\
    H^{p+q-1}(X\times U,n)\ar[r]^{\partial}\ar[d]^{\cl^{X\times U}(\Gamma\times Y|_{U\times Y})}&H^{p+q}_{X\times Z}(X\times Y,n)\ar[d]^{\cl^{X\times Y}(\Gamma\times Y)}\ar[r]^{r}&H^{p+q}_{X\times Z_{\rm sm}}(X\times V,n)\ar[d]^{\cl^{X\times V}(\Gamma\times Y)}\\
    H^{p+q-1+2c}_{\Gamma\times Y}(X\times U,n+c)\ar[d]^{\rm exc}_{\sim}&H^{p+q+2c}_{\Gamma\times Z}(X\times Y,n+c)\ar@{=}[d]\ar[r]^{r}&H^{p+q+2c}_{\Gamma\times Z_{\rm sm}}(X\times V,n+c)\ar@{=}[d]\\
    H^{p+q-1+2c}_{\Gamma\times Y}(X\times Y \setminus (\Gamma\times Z),n+c)\ar[r]^{\hspace{10mm}\partial}\ar[d]^{\iota_*}&H^{p+q+2c}_{\Gamma\times Z}(X\times Y,n+c)\ar@{=}[d]\ar[r]^{r}&H^{p+q+2c}_{\Gamma\times Z_{\rm sm}}(X\times V,n+c)\ar@{=}[d]\\
    H^{p+q-1+2c}(X\times Y \setminus (\Gamma\times Z),n+c)\ar[r]^{\hspace{10mm}\partial}&H^{p+q+2c}_{\Gamma\times Z}(X\times Y,n+c)\ar[r]^{r}&H^{p+q+2c}_{\Gamma\times Z_{\rm sm}}(X\times V,n+c)
 }\tag{Diagram \ref{prop-exact sequence is compatible with cross product} I}
 }
\end{equation*}
 \begin{equation*}\label{diagram-exact sequence is compatible with cross product2}
\adjustbox{scale=.7}{
\xymatrix{H^{p+q}_{Z_{\rm sm}}(V,n)\ar[d]^{p^*}&H^{q-p}(Z_{\rm sm},n-p)\ar[l]_{\sim}\ar[d]^{\sum_i (m_ip_i^*)}\ar@{=}[r]&H^{q-p}(Z_{\rm sm},n-p)\ar[d]_{\sum_i (m_ip_i^*)}\ar@{=}[r]&H^{q-p}(Z_{\rm sm},n-p)\ar[d]_{\sum_i (m_ip_i^*)}\\
    H^{p+q}_{X\times Z_{\rm sm}}(X\times V,n)\ar[d]^{\cl^{X\times V}(\Gamma\times Y)}&H^{q-p}(X\times Z_{\rm sm},n-p)\ar[d]^{\cl^{X\times Z_{\rm sm}}(\Gamma\times Z_{\rm sm})}\ar[l]_{\sim}\ar[r]^{r}&H^{q-p}(X\setminus \Gamma_{\rm sing}\times Z_{\rm sm},n-p)\ar[d]_{\cl^{X\setminus \Gamma_{\rm sing}\times Z_{\rm sm}}(\Gamma\times Z_{\rm sm})}\ar[r]^{i^*}&H^{q-p}(\Gamma_{\rm sm}\times Z_{\rm sm},n-p)\ar[d]^{i_*}\\
    H^{p+q+2c}_{\Gamma\times Z_{\rm sm}}(X\times V,n+c)&H^{q-p+2c}_{\Gamma\times Z_{\rm sm}}(X\times Z_{\rm sm},n+c-p)\ar[l]_{\hspace{-4mm}\sim}\ar[r]^{\hspace{-5mm}r}&H^{q-p+2c}_{\Gamma_{\rm sm}\times Z_{\rm sm}}(X\setminus \Gamma_{\rm sing}\times Z_{\rm sm},n+c-p)\ar@{=}[r]&H^{q-p+2c}_{\Gamma_{\rm sm}\times Z_{\rm sm}}(X\setminus \Gamma_{\rm sing}\times Z_{\rm sm},n+c-p)
}
}\tag{Diagram \ref{prop-exact sequence is compatible with cross product} II}
\end{equation*}
\begin{equation*}    \label{diagram-exact sequence is compatible with cross product3}
\adjustbox{scale=.9}{
\xymatrix{H^{p+q}_{\overline{\{y\}}\cap U}(U,n)\ar[r]^{\iota_*}\ar[d]^{\pi^*}&H^{p+q}(U,n)\ar[d]^{\pi^*}\\
H^{p+q}_{\Gamma\times ({\overline{\{y\}}\cap U})}(\Gamma\times U,n)\ar[r]^{\iota_*}\ar[d]^{\rm Gysin}&H^{p+q}(\Gamma\times U,n)\ar[d]^{\rm Gysin}\\
H^{p+q+2c}_{\Gamma\times ({\overline{\{y\}}\cap U})}(X\times U,n+c)\ar[r]^{\iota_*}\ar[d]^{\rm exc}_{\sim}&H^{p+q+2c}_{\Gamma\times U}(X\times U,n+c)\ar[d]^{\rm exc}_{\sim}\\
H^{p+q+2c}_{\Gamma\times ({\overline{\{y\}}\cap U})}(X\times Y\setminus(\Gamma\times Z),n+c)\ar[r]^{\iota_*}\ar@{=}[d]&H^{p+q+2c}_{\Gamma\times U}(X\times Y\setminus(\Gamma\times Z),n+c)\ar[d]^{\iota_*}\\
 H^{p+q+2c}_{\Gamma\times ({\overline{\{y\}}\cap U})}(X\times Y\setminus(\Gamma\times Z),n+c)\ar[r]^{\iota_*}&H^{p+q+2c}(X\times Y\setminus(\Gamma\times Z),n+c),
    }\tag{Diagram \ref{prop-exact sequence is compatible with cross product}III}
    }
\end{equation*}
Now, we are going to show the right one is commutative as well. For this purpose, let us apply our second definition for external/cross product, i.e., \Cref{def-external/cross product II}. Noting that $\iota_*$ is defined on the level of the first page $E_1^{p,q}$ of the Bloch--Ogus spectral sequence (see \cite[Proposition 7.35]{schreieder2020infinite}) and the middle vertical arrow is defined on the level of page $1$ as well, it suffices to show that the correspondence products at the first page are compatible. For any $\alpha\in E_1^{p,q}(Y,n)$ which only has non-zero component at $y\in Y^{(p)}$, we can find a closed subset $Z\subset \overline{\{y\}}$, containing the singularities of $\overline{\{y\}}$, such that $\alpha\in H^{q-p}(\overline{\{y\}}\setminus Z,n-p)\cong H^{p+q}_{\overline{\{y\}}\cap U}(U,n)$ with $U\coloneqq Y\setminus Z$ by the purity isomorphism. For the theories of cohomology with supports, we indeed have the commutative \ref{diagram-exact sequence is compatible with cross product3}, where $\pi\colon  \Gamma\times Y\to Y$ is the projection. Now, change the cohomology of supports on the left hand side to cohomology via the Gysin homomorphism (which is the purity isomorphism for smooth pair) after taking limits; we get the following commutative \ref{diagram-exact sequence is compatible with cross product4}, where $u_i(1\leq i\leq s)$ is the generic/maximal points of $\Gamma\times ({\overline{\{y\}}\cap U})$, $n_i\coloneqq {\rm length}(\mathcal{O}_{\Gamma\times ({\overline{\{y\}}\cap U}),u_i})$ and $\pi_*$ is the induced map $\overline{\{u_i\}}\to U$ along $\pi$. 

Recalling the definition of cross product of cycle modules in \cite[Section 14]{Ros96}, this commutative diagram tells us that for such $\alpha\in H^{q-p}(y,n-p)\subset E_1^{p,q}(Y,n)$, the element $[\Gamma]\times \alpha$ can be represented by an element in $H^{p+q+2c}_{\Gamma\times (\overline{\{y\}}\cap U)}(X\times U,n+c)$, which can be also viewed as an element in $H^{p+q+2c}_{\Gamma\times ({\overline{\{y\}}\cap U})}(X\times Y\setminus(\Gamma\times Z),n+c)$ (note that excision does not change $E_1^{p,q}$); say $\beta$. And by \ref{diagram-exact sequence is compatible with cross product3},  the $\alpha$ maps as follows
\begin{equation*}
\xymatrix{\alpha\ar@{|->}[rr]^{\iota_*\phantom{dddd}}\ar@{|->}[d]&&\omega\coloneqq \iota_*\alpha\ar@{|->}[d]\\
    \beta\ar@{|->}[rr]^{\iota_*\phantom{dddd}}&&\theta\coloneqq \iota_*\beta
    }
\end{equation*}
where the right arrow is given by \Cref{def-external/cross product II}, hence it is commutative by the commutativity of \ref{diagram-exact sequence is compatible with cross product3}. Therefore, since the cross product is bi-additive, we get the commutativity of the right square as well.

\begin{equation*}\label{diagram-exact sequence is compatible with cross product4}
\adjustbox{scale=.9}{
\xymatrix{H^{q-p}({\overline{\{y\}}\cap U},n-p)\ar[rr]^{\sim}_{\rm Gysin}\ar[d]^{\sum_i (n_i\pi_i^*)}&&H^{p+q}_{\overline{\{y\}}\cap U}(U,n)\ar[d]^{\pi^*}\ar@{=}[r]&H^{p+q}_{\overline{\{y\}}\cap U}(U,n)\ar[dd]^{\pi^*}\\
H^{q-p}(\Gamma\times ({\overline{\{y\}}\cap U}),n-p)\ar[d]^{\rm restr.}&&H^{p+q}_{\Gamma\times (\overline{\{y\}}\cap U)}(\Gamma\times U,n)\ar[d]^{\rm restr.}&\\
H^{q-p}(\Gamma_{\rm sm}\times ({\overline{\{y\}}\cap U}),n-p)\ar[d]^{\rm limit}&&H^{p+q}_{\Gamma_{\rm sm}\times (\overline{\{y\}}\cap U)}(\Gamma_{\rm sm}\times U,n+c)\ar[ll]^{\sim}_{\rm (purity/Gysin)^{-1}}\ar[d]^{\sim}_{\rm purity/Gysin}&H^{p+q}_{\Gamma\times (\overline{\{y\}}\cap U)}(\Gamma\times U,n)\ar[l]^{\phantom{dddd}\rm restr.}\ar[d]^{\rm Gysin}\\
\bigoplus_{u_i}H^{q-p}(u_i,n-p)&&H^{p+q+2c}_{\Gamma_{\rm sm}\times (\overline{\{y\}}\cap U)}(X\setminus \Gamma_{\rm sing}\times U,n
+c)\ar[ll]^{\hspace{-6mm}\text{limit }\circ \text{ purity}^{-1}}&H^{p+q+2c}_{\Gamma\times (\overline{\{y\}}\cap U)}(X\times U,n+c)\ar[l]^{\hspace{8mm}\rm restr.}
}
}\tag{Diagram \ref{prop-exact sequence is compatible with cross product}IV}
\end{equation*}
\end{proof}

Next we show that the cross-product is compatible with pull-back and push-forward maps.
\begin{lem}\label{lem-cross product for proper push}
Let $f\colon X\to Y$ and $g\colon Z\to W$ be proper morphisms between smooth schemes. Then $(f\times g)_*([\Gamma]\times [\alpha])=f_*[\Gamma]\times g_*[\alpha]$ for any $[\Gamma]\in \CH^c(X)$ and any $[\alpha]\in H^i_{j,\nr}(Z,n)$.
\end{lem}
\begin{proof}
    For any $[\alpha]\in H^i_{j,\nr}(Z,n)$, we can find an open subset $U\subset Z$ such that $\alpha\in H^i(U,n)$ and $\codim(R)=j+2$, where $R\coloneqq Z\setminus U$. Then $g_*[\alpha]\in H^{i+2d'}(W\setminus g(R),n+d')$ is the image of $\alpha$ under the following 
    \begin{equation*}
      \phi \colon  H^i(U,n)\xrightarrow{\rm restr.}H^i(Z\setminus (g^{-1}(g(R))),n)\xrightarrow{g_*} H^{i+2d'}(W\setminus g(R),n+d'),
    \end{equation*}
    where $d'\coloneqq \dim(W)-\dim(Z)$. Similarly, we can define the morphism  $\xi$, such that $(f\times \id)_*([\beta])=[\xi(\beta)]\in H^{*+2d}_{\dagger+d,\nr}(Y\times W,\ddagger+d)$ for every $\beta\in H^*_{\dagger,\nr}(X\times W,\ddagger)$, where $d\coloneqq \dim(Y)-\dim(X)$. Write $|\Gamma|\coloneqq {\rm Supp}(\Gamma)$. Now, this lemma is given by the commutative diagram \ref{diagram-lem-cross product for proper push}, where $m\coloneqq c+d+d'$, $p\colon X\times Z\to Z$, $q\colon Y\times W\to W$ and $\pi\colon X\times W\to W$. Noting that $$(f\times \id)_*(\cl^{X\times (W\setminus g(R))}_{|\Gamma|\times (W\setminus g(R))}(\Gamma\times (W\setminus g(R))))=\cl^{Y\times (W\setminus g(R))}_{f|\Gamma|\times (W\setminus g(R))}(f_*(\Gamma)\times (W\setminus g(R))),$$
    every square is commutative by functorial properties and/or projection formulas for cup products of cohomology with supports.

    Note that the image of $\alpha$ under the composition of all left vertical arrows and all bottom horizontal arrows is the element representing $(\id\times g)_*\circ (f\times \id)_*([\Gamma]\times [\alpha])=(f\times g)_*([\Gamma]\times [\alpha])$, and the image of $\alpha$ under composition of all top arrows and all right arrows is the element representing $f_*[\Gamma]\times g_*[\alpha]$. Hence $(f\times g)_*([\Gamma]\times [\alpha])=f_*[\Gamma]\times g_*[\alpha]$ for any $[\Gamma]\in \CH^c(X)$ and any $[\alpha]\in H^i_{j,\nr}(Z,n)$.
      \begin{equation*}\label{diagram-lem-cross product for proper push}
    \adjustbox{width=\textwidth}{
      \hspace{-9pt}  \xymatrix{H^i(U,n)\ar[rr]^{\phi\phantom{ddxxxxd}}\ar[d]^{p^*}&&H^{i+2d'}(W\setminus g(R),n+d')\ar@{=}[rr]\ar[d]^{\pi^*}&&H^{i+2d'}(W\setminus g(R),n+d')\ar[d]^{q^*}\\
        H^i(X \times U,n)\ar[rr]^{\phi\phantom{ddd}}\ar[d]^{\cl^{X\times U}_{|\Gamma|\times U}(\Gamma\times U)}&&H^{i+2d'}(X\times (W\setminus g(R)),n+d')\ar[d]_{\cl^{X\times (W\setminus g(R))}_{|\Gamma|\times (W\setminus g(R))}(\Gamma\times (W\setminus g(R)))}&&H^{i+2d'}(Y\times (W\setminus g(R)),n+d')\ar[d]_{(f\times \id)_*(\cl^{X\times (W\setminus g(R))}_{|\Gamma|\times (W\setminus g(R))}(\Gamma\times (W\setminus g(R))))}\ar[ll]_{(f\times \id)^*}\\
        H^{i+2c}_{|\Gamma|\times U}(X \times U,n+c)\ar[rr]^{\phi\phantom{ddd}}\ar[d]^{\sim}_{\rm exc}&&H^{i+2d'+2c}_{|\Gamma|\times (W\setminus g(R))}(X\times (W\setminus g(R)),n+d'+c)\ar[rr]^{\xi}\ar[d]^{\sim}_{\rm exc}&&H^{i+2m}_{f(|\Gamma|)\times (W\setminus g(R))}(Y\times (W\setminus g(R)),n+m)\ar[d]^{\sim}_{\rm exc}\\
        H^{i+2c}_{|\Gamma|\times U}(X \times Z\setminus (|\Gamma|\times R),n+c)\ar[rr]^{\phi\phantom{ddddd}}\ar[d]^{\iota_*}&&H^{i+2d'+2c}_{|\Gamma|\times (W\setminus g(R))}(X\times W\setminus(|\Gamma|\times g(R)),n+d'+c)\ar[rr]^{\xi}\ar[d]^{\iota_*}&&H^{i+2m}_{f(|\Gamma|)\times (W\setminus g(R))}(Y\times W\setminus(f(|\Gamma|)\times g(R)),n+m)\ar[d]^{\iota_*}\\
         H^{i+2c}(X \times Z\setminus (|\Gamma|\times R),n+c)\ar[rr]^{\phi\phantom{dddddd}}&&H^{i+2d'+2c}(X\times W\setminus(|\Gamma|\times g(R)),n+d'+c)\ar[rr]^{\xi}&&H^{i+2m}(Y\times W\setminus(f(|\Gamma|)\times g(R)),n+m)
        }}\tag{Diagram \ref{lem-cross product for proper push}}
    \end{equation*}
\end{proof}
\begin{lem}\label{lem-cross product for pull back}
 Let $f\colon X\to Y$ and $g\colon Z\to W$ be morphisms between smooth schemes. Then $(f\times g)^*([\Gamma]\times [\alpha])=f^*[\Gamma]\times g^*[\alpha]$ for any $[\Gamma]\in \CH^c(Y)$ and any $[\alpha]\in H^i_{j,\nr}(W,n)$.   
\end{lem}
\begin{proof}
Let $[\Gamma]\in \CH^c(Y)$ and $[\alpha]\in H^i_{j,\nr}(W,n)$. This lemma follows if we can show the following two equations:
\begin{equation*}
    \begin{split}
        (\id \times g)^*([\Gamma]\times [\alpha])=[\Gamma]\times g^*[\alpha],\\
        (f\times \id)^*([\Gamma]\times [\alpha])=f^*[\Gamma]\times [\alpha].
    \end{split}
\end{equation*}
After all, using \Cref{thm-functoriality}, these equations imply $(f\times g)^*([\Gamma]\times [\alpha])=(f\times \id)^*([\Gamma]\times g^*[\alpha])=f^*[\Gamma]\times g^*[\alpha]$. 

Furthermore, we may assume that both $f$ and $g$ are regular embeddings, as the equations hold true for flat pullbacks already.

First, let us show $(\id \times g)^*([\Gamma]\times [\alpha])=[\Gamma]\times g^*[\alpha]$. For the given $[\alpha]\in H^i_{j,\nr}(W,n)$, there exists an open subset $U\subset W$ such that $\codim(R)=j+2$ and $\alpha\in H^i(U,n)$, where $R\coloneqq W\setminus U$. Let $\overline{R}$ be the closure of $R\times \G_m$ in $D(W,Z)$. Write $|\Gamma|\coloneqq {\rm Supp}(\Gamma)$ and remember that for flat pullbacks this equation holds naturally. Now, we have the following commutative diagram 
\begin{equation*}
\adjustbox{scale=.6}{
    \xymatrix{H^i(U\times \G_m,n)\ar[rr]^{\{t\}}\ar[d]^{p^*}&&H^{i+1}(U\times \G_m,n+1)\ar[rr]^{-\partial}\ar[d]^{p^*}&&H^{i+2}_{N_ZW\setminus \overline{R}}(D(W,Z)\setminus \overline{R},n+1)\ar[d]^{p^*}\\
    H^i(Y\times U\times \G_m,n)\ar[rr]^{\{t\}}\ar[d]^{\cl^{Y\times U\times \G_m}_{|\Gamma|\times U\times \G_m}(\Gamma\times U\times \G_m)}&&H^{i+1}(Y\times U\times \G_m,n+1)\ar[rr]^{-\partial}\ar[d]^{\cl^{Y\times U\times \G_m}_{|\Gamma|\times U\times \G_m}(\Gamma\times U\times \G_m)}&&H^{i+2}_{Y\times N_ZW\setminus \overline{R}}(Y\times D(W,Z)\setminus \overline{R},n+1)\ar[d]^{\cl^{Y\times D(W,Z)\setminus \overline{R}}_{|\Gamma|\times N_ZW\setminus \overline{R}}(\Gamma\times N_ZW\setminus \overline{R})}\\
   H^{i+2c}_{|\Gamma|\times U\times \G_m}(Y\times U\times \G_m,n+c)\ar[rr]^{\{t\}}\ar[d]^{\sim}_{\rm exc}&&H^{i+1+2c}_{|\Gamma|\times U\times \G_m}(Y\times U\times \G_m,n+1+c)\ar[rr]^{-\partial}\ar[d]^{\sim}_{\rm exc}&&H^{i+2+2c}_{|\Gamma|\times N_ZW\setminus \overline{R}}(Y\times D(W,Z)\setminus \overline{R},n+1+c)\ar[d]^{\sim}_{\rm exc}\\
    H^{i+2c}_{|\Gamma|\times U\times \G_m}((Y\times W\setminus |\Gamma|\times R)\times \G_m,n+c)\ar[rr]^{\{t\}}\ar[d]^{\iota_*}&&H^{i+1+2c}_{|\Gamma|\times U\times \G_m}((Y\times W\setminus |\Gamma|\times R)\times \G_m,n+1+c)\ar[rr]^{-\partial}\ar[d]^{\iota_*}&&H^{i+2+2c}_{|\Gamma|\times N_ZW\setminus \overline{R}}(Y\times D(W,Z)\setminus |\Gamma|\times \overline{R},n+1+c)\ar[d]^{\iota_*}\\
    H^{i+2c}((Y\times W\setminus |\Gamma|\times R)\times \G_m,n+c)\ar[rr]^{\{t\}}&&H^{i+1+2c}((Y\times W\setminus |\Gamma|\times R)\times \G_m,n+1+c)\ar[rr]^{-\partial}&&H^{i+2+2c}_{Y\times N_ZW\setminus |\Gamma|\times \overline{R}}(Y\times D(W,Z)\setminus |\Gamma|\times \overline{R},n+1+c)\nospacepunct{,}
    }}
\end{equation*}
where $p^*$ is induced by projection from $Y$. The square in the second row and second column commutes due to the following well-known fact: let $Z\subset W\subset X$ be closed subsets, then we have the following commutative diagram
\begin{equation*}
 \xymatrix{H^i_{Z}(X,n)\ar[r]^{\iota_*}\ar[d]^{\cup \alpha}&H^i_W(X,n)\ar[r]^{\hspace{-7mm}\rm restr.}\ar[d]^{\cup \alpha}&H^i_{W\setminus Z}(X\setminus Z,n)\ar[r]^{\partial}\ar[d]^{\cup \alpha|_{X\setminus Z}}&H^{i+1}_Z(X,n)\ar[d]^{\cup \alpha}\\
H^{i+j}_{Z\cap Y}(X,n+m)\ar[r]^{\iota_*}&H^{i+j}_{W\cap Y}(X,n+m)\ar[r]^{\hspace{-7mm}\rm restr.}&H^{i+j}_{(W\setminus Z)\cap Y}(X\setminus Z,n+m)\ar[r]^{\partial}&H^{i+1+j}_{Z\cap Y}(X,n+m)\nospacepunct{,}
  }   
\end{equation*}
for any $\alpha\in H^j_{Y}(X,\Z/\ell (m))$. 

Combining the above diagram with the flat pullbacks $H^i(U,n)\to H^i(U\times \G_m,n)$ and $H^{i+2c}(Y\times W\setminus |\Gamma|\times R,n+c)\to H^{i+2c}((Y\times W\setminus |\Gamma|\times R)\times \G_m,n+c)$ and applying compatibility of the purity isomorphisms, we get the following commutative diagram
\begin{equation*}
    \xymatrix{H^i(U,n)\ar[rr]^{g^*}\ar[d]^{\Gamma \times}&&H^i(N_ZW\setminus \overline{R},n)\ar[d]^{\Gamma \times} \\
    H^{i+2c}(Y\times W\setminus |\Gamma|\times R,n+c)\ar[rr]^{(\id\times g)^*}&&H^{i+2c}(Y\times N_ZW\setminus |\Gamma|\times \overline{R},n+c)\nospacepunct{.}
    }
\end{equation*}
Therefore, we have $(\id \times g)^*([\Gamma]\times [\alpha])=[\Gamma]\times g^*[\alpha]$.

It remains to show $(f \times \id)^*([\Gamma]\times [\alpha])=f^*[\Gamma]\times [\alpha]$. Let us use the same notations as above. Let $[\overline{\Gamma}]\in \CH^c(D(X,Y))$ be the closure of $[\Gamma\times \G_m]\in \CH^c(Y\times \G_m)$ and write $\overline{|\Gamma|}$ as the closure of $|\Gamma|\times \G_m$ in $D(X,Y)$ so that $\overline{|\Gamma|}={\rm Supp}(\overline{\Gamma})$. In this case, we have the following commutative diagram
\begin{equation*}
\adjustbox{scale=.6}{
    \xymatrix{H^i(U\times \G_m,n)\ar[rr]^{\{t\}}\ar[d]^{p^*}&&H^{i+1}(U\times \G_m,n+1)\ar[rr]^{-\partial}\ar[d]^{p^*}&&H^{i+2}_{U\times \{0\}}(U\times \mathbb{A}^1,n+1)\ar[d]^{p^*}\\
    H^i(Y\times U\times \G_m,n)\ar[rr]^{\{t\}}\ar[d]^{\cl^{Y\times U\times \G_m}_{|\Gamma|\times U\times \G_m}(\Gamma\times U\times \G_m)}&&H^{i+1}(Y\times U\times \G_m,n+1)\ar[rr]^{-\partial}\ar[d]^{\cl^{Y\times U\times \G_m}_{|\Gamma|\times U\times \G_m}(\Gamma\times U\times \G_m)}&&H^{i+2}_{N_XY\times U}(D(X,Y)\times U,n+1)\ar[d]^{\cl^{D(X,Y)\times U}_{(\overline{|\Gamma|}\cap N_XY)\times U}((\overline{\Gamma}|_{N_XY})\times U)}\\
   H^{i+2c}_{|\Gamma|\times U\times \G_m}(Y\times U\times \G_m,n+c)\ar[rr]^{\{t\}}\ar[d]^{\sim}_{\rm exc}&&H^{i+1+2c}_{|\Gamma|\times U\times \G_m}(Y\times U\times \G_m,n+1+c)\ar[rr]^{-\partial}\ar[d]^{\sim}_{\rm exc}&&H^{i+2+2c}_{(\overline{|\Gamma|}\cap N_XY)\times U}(D(X,Y)\times U, n+1+c)\ar[d]^{\sim}_{\rm exc}\\
    H^{i+2c}_{|\Gamma|\times U\times \G_m}((Y\times W\setminus |\Gamma|\times R)\times \G_m,n+c)\ar[rr]^{\{t\}}\ar[d]^{\iota_*}&&H^{i+1+2c}_{|\Gamma|\times U\times \G_m}((Y\times W\setminus |\Gamma|\times R)\times \G_m,n+1+c)\ar[rr]^{-\partial}\ar[d]^{\iota_*}&&H^{i+2+2c}_{(\overline{|\Gamma|}\cap N_XY)\times U}(D(Y,X)\times W\setminus \overline{|\Gamma|}\times R,n+1+c)\ar[d]^{\iota_*}\\
    H^{i+2c}((Y\times W\setminus |\Gamma|\times R)\times \G_m,n+c)\ar[rr]^{\{t\}}&&H^{i+1+2c}((Y\times W\setminus |\Gamma|\times R)\times \G_m,n+1+c)\ar[rr]^{-\partial}&&H^{i+2+2c}_{N_XY\times W\setminus \overline{|\Gamma|}\times R}(D(Y,X)\times W\setminus \overline{|\Gamma|}\times R,n+1+c)\nospacepunct{,}
    }}
\end{equation*}
where the commutativity of square in the second row and second column is also due to the compatibility of cup product and localization sequences. Note that $[\overline{\Gamma}|_{N_XY}]$ is exactly the image of $f^*[\Gamma]$ under the flat pull-back $\pi^*\colon \CH^c(X)\to \CH^c(N_XY)$ by \cite[Chapter~5]{fulton2013intersection}; and the composition of $H^i(U,n)\to H^i(U\times \G_m,n)$ with the top arrows together with the purity isomorphism gives exactly the identity $\id\colon H^i(U,n)\to H^i(U,n)$. Similar to the above, we also get a commutative diagram
\begin{equation*}
    \xymatrix{H^i(U,n)\ar@{=}[rr]\ar[d]^{\Gamma\times}&&H^i(U,n)\ar[d]^{f^*\Gamma\times}\\
    H^{i+2c}(Y\times W\setminus |\Gamma|\times R,n+c)\ar[rr]^{(f\times \id)^*}&&H^{i+2c}(N_XY\times W\setminus \overline{|\Gamma|}\times R,n+c)\nospacepunct{,}
    }
\end{equation*}
and hence $(f \times \id)^*([\Gamma]\times [\alpha])=f^*[\Gamma]\times [\alpha]$.
\end{proof}
\begin{prop}\label{prop-commutes with external product}
  Let $X$ and $Y$ be smooth equi-dimensional quasi-projective schemes defined over a field $k$, and $M\coloneqq \Z/m\Z$, where $m$ is an integer invertible in $k$. Then there is a commutative diagram of exact sequences
    \begin{equation*}
    \adjustbox{scale=.75}{
        \xymatrix{\CH^b(Y,n;M)\ar[r]^{\mathrm{cl}}\ar[d]^{[\Gamma]\times }&H^{2b-n}(Y,M(b))\ar[r]\ar[d]^{[\Gamma]\times }& H^{2b-n}_{b-n-1,\nr}(Y,M(b))\ar[r]^{\theta}\ar[d]^{[\Gamma]\times }&\CH^b(Y,n-1;M)\ar[d]^{[\Gamma]\times }\\
        \CH^{b+c}(X\times Y,n;M)\ar[r]^{\mathrm{cl}}&H^{2b-n+2c}(X\times Y,M(b+c))\ar[r]& H^{2b-n+2c}_{b-n-1+c,\nr}(X\times Y,M(b+c))\ar[r]^{\theta}&\CH^{b+c}(X\times Y,n-1;M)
        }}
    \end{equation*}   
for every $[\Gamma]\in \CH^c(X)$.
\end{prop}
\begin{proof}
  It suffices to show that the last square commutes. For every $\alpha\in H^{2b-n}_{b-n-1,\nr}(Y,M(b))$, we can find an open subset $U$ such that $Z\coloneqq Y\setminus U$ has pure codimension $b-n+1$ and $\alpha\in H^{2b-n}(U,M(b))$. By the definition of $H^{2b-n}_{b-n-1,\nr}(Y,M(b))\to \CH^b(Y,n-1;M)$ (see \cite[(27)]{kok2023higher} or the proof of \Cref{prop-Prop 2.15}), $\alpha$ maps into $\CH^b(Y,n-1;M)$ under the following morphism
  \begin{equation*}
      H^{2b-n}(U,b)\xrightarrow{\partial} H^{2b-n+1}_Z(Y,b)\xleftarrow{\sim}H^{n-1}_{BM}(Z,n-1)\xleftarrow{\sim}\CH^{n-1}(Z,n-1;M)\xrightarrow{\iota_*}\CH^b(Y,n-1;M),
  \end{equation*}
  where the second isomorphism is given by \cite[Proposition 4.9]{kok2023higher}. Noting that except for the first arrow, the commutativity with external product is already known. We finish our proof as the first arrow is compatible with external product by \ref{diagram-exact sequence is compatible with cross product1}.
\end{proof}
To end this subsection, let us introduce the \emph{action of cycles}.
\begin{define}[Action of Cycles]\label{def-action of cycles on refined unramified cohomology}
    Let $X$ be a smooth equi-dimensional algebraic scheme. We have a well-defined pairing 
    \begin{equation*}
        \CH^c(X)\otimes H^i_{j,\nr}(X,n)\xrightarrow{\times }H^{i+2c}_{j+c,\nr}(X\times X,n+c)\xrightarrow{\Delta^*} H^{i+2c}_{j+c,\nr}(X,n+c),
    \end{equation*}
    where the first arrow is the cross product and the last one is the pull-back along the diagonal. Moreover, we write the image of $(\Gamma,\alpha)$ as $\Gamma\cdot \alpha$.
\end{define}
\begin{lem}[Projection Formula II]\label{lem-projection formula II}
   Let $f\colon X\to Y$ be a morphism between smooth equi-dimensional algebraic schemes. For any $\Gamma \in \CH^c(Y)$ and $\alpha\in H^i_{j,\nr}(Y)$, $f^*(\Gamma\cdot\alpha)=(f^*\Gamma)\cdot(f^*\alpha)$. Moreover, if $f$ is proper, then we have the projection formulas:
   \begin{itemize}
       \item[(i)] for any $\Gamma\in \CH^c(X)$ and $\alpha\in H^i_{j,\nr}(Y,n)$, $f_*(\Gamma\cdot f^*(\alpha))=f_*(\Gamma)\cdot\alpha$;
       \item[(ii)] for any $\Gamma\in \CH^c(Y)$ and $\alpha\in H^i_{j,\nr}(X,n)$, $f_*(f^*(\Gamma)\cdot \alpha)=\Gamma\cdot f_*(\alpha)$. 
   \end{itemize}  
\end{lem}
\begin{proof}
   Let us first show that $f^*(\Gamma\cdot\alpha)=(f^*\Gamma)\cdot(f^*\alpha)$. Applying \Cref{thm-functoriality} and \Cref{lem-cross product for pull back}, we have 
   \begin{eqnarray*}
      f^*(\Gamma\cdot\alpha)\coloneqq f^*\Delta^*_Y(\Gamma\times \alpha)=\Delta_X^*(f\times f)^*(\Gamma\times \alpha)=\Delta_X^*(f^*\Gamma\times f^*\alpha)=(f^*\Gamma)\cdot(f^*\alpha).
   \end{eqnarray*}
   Now, suppose that $f$ is proper. Then (i) is given by the equation
   \begin{eqnarray*}
      f_*(\Gamma)\cdot\alpha=\Delta_Y^*(f_*\Gamma\times \alpha)=\Delta_Y^*(f\times \id_Y)_*(\Gamma\times \alpha)=f_*\Delta_X^*(\id_X\times f)^*(\Gamma\times \alpha)=f_*(\Gamma\cdot f^*(\alpha))
   \end{eqnarray*}
   and (ii) is given by 
   \begin{eqnarray*}
       \Gamma\cdot f_*(\alpha)=\Delta_Y^*(\Gamma\times f_*(\alpha))=\Delta_Y^*(\id_Y\times f)_*(\Gamma\times \alpha)=f_*\Delta_X^*(f\times \id_X)^*(\Gamma\times \alpha)=f_*(f^*(\Gamma)\cdot \alpha),
   \end{eqnarray*}
   by applying \Cref{lem-projection formual for refined}, \Cref{lem-cross product for proper push} and \Cref{lem-cross product for pull back}.
\end{proof}

\subsection{Correspondence Action}
In this section, we define the correspondence action on refined unramified cohomology and show this action is compatible with the action on the Bloch--Ogus spectral sequences (\ref{equ-corr actions 1}). This section is inspired by \cite{fulton2013intersection} and \cite{schreieder2022moving}.
\begin{define}\label{new-def-corr-action}
Let $X$, $Y$ and $Z$ be smooth proper equi-dimensional $k$-schemes. We define the correspondence action
\[
\CH^c(X\times Y)\otimes H^p_{q,\nr}(Y\times Z,n) \to H^{p+2c-2d_Y}_{q+c-d_Y,\nr}(X\times Z, n+c-d_Y)\colon[\Gamma]\otimes\alpha\mapsto[\Gamma]_\ast\alpha,
\]
as the following composition
\begin{align*}
\CH^c(X\times Y)\otimes H^p_{q,\nr}(Y\times Z,n)\xrightarrow{p_{XY}^\ast\otimes p_{YZ}^\ast}\CH^c(X\times Y\times X)\otimes H^p_{q,\nr}(X\times Y\times Z,n)\\
\xrightarrow{\times}H^{p+2c}_{q+c}(X\times Y\times Z\times X\times Y\times Z,n+c)\xrightarrow{\Delta_{XYZ}^\ast}H^{p+2c}_{q+c,\nr}(X\times Y\times Z,n+c)\\
\xrightarrow{(p_{XZ})_\ast}H^{p+2c-2d_Y}_{q+c-d_Y,\nr}(X\times Z,n+c-d_Y).
\end{align*}
\end{define}

\begin{rem}
\rm{As in \Cref{remark compatibility}, now using \Cref{lem-cross product for proper push} and \Cref{lem-cross product for pull back}, the above construction is the same as the composition
\begin{align*}
&\CH^c(X\times Y)\otimes H^p_{q,\nr}(Y\times Z,n)\xrightarrow{\times}H^{p+2c}_{q+c,\nr}(X\times Y\times Y\times Z,n+c)\\
&\xrightarrow{(\id_X\times \Delta_Y\times\id_Z)^\ast}H^{p+2c}_{q+c,\nr}(X\times Y\times Z,n+c)\xrightarrow{(p_{XZ})_\ast}H^{p+2c-2d_Y}_{q+c-d_Y,\nr}(X\times Z,n+c-d_Y).
\end{align*}
}
\end{rem}

\begin{prop}[{cf.~\cite[Corollary 6.8]{schreieder2022moving}}]\label{prop-actions of corr}
The above construction defines a bi-additive pairing with the following properties:
\begin{itemize}
    \item[(1)] If $q'\leq q$, then the following diagram commutes    
    \begin{equation*}
        \xymatrix{\CH^c(X\times Y)\times H^p_{q,\nr}(Y\times Z,n)\ar[r]\ar[d]& H^{p+2c-2d_Y}_{q+c-d_Y,\nr}(X\times Z,n+c-d_Y)\ar[d]\\
        \CH^c(X\times Y)\times H^p_{q',\nr}(Y\times Z,n)\ar[r]& H^{p+2c-2d_Y}_{q'+c-d_Y,\nr}(X\times Z,n+c-d_Y)\nospacepunct{,}
        }
    \end{equation*}
    where the vertical maps are induced by the restriction maps.
    \item[(2)] If $q\geq \lceil p/2\rceil$, this pairing is the same as the usual action of correspondence on cohomology groups. 
    \item[(3)] Let $W$ be a smooth projective equi-dimenisonal $k$-scheme of dimension $d_W$ and let $[\Gamma']\in \CH^{c'}(W\times X)$. For any $[\alpha]\in H^p_{q,\nr}(Y\times Z,n)$,
    \begin{equation*}
        [\Gamma']_*([\Gamma]_*[\alpha])=([\Gamma']\circ[\Gamma])_*[\alpha]\in H^{p+2c+2c'-2d_X-2d_Y}_{q+c+c'-d_X-d_Y}(W\times Z,n+c+c'-d_X-d_Y),
    \end{equation*}
    where $[\Gamma']\circ[\Gamma]$ is the composition of the correspondences.
\end{itemize}
\end{prop}
\begin{proof}
Item \emph{(1)} follows from an argument like the proof of  \Cref{lem-cross product for pull back}.

Now \emph{(2)} follows because on ordinary cohomology, the pull-back map defined in \Cref{def-pullback along lci} is the same as the ordinary pull-back (see \Cref{pull-back ordinary cohomology}). And the cross-product of \Cref{def-external/cross product I} boils down to $\Gamma \times\alpha=p_X^\ast\Gamma\cup p_Y^\ast\alpha$, which distributes along pull-back maps.

Lastly, \emph{(3)} follows from the following computations. First of all, we have
\[
\Gamma'\circ\Gamma = (p_{WY})_\ast(p_{WX}^\ast\Gamma'\cdot p_{XY}^\ast\Gamma),
\]
where we write the map $p_\dagger$ for the projections from $W\times X\times Y$. Then if we write $q_\dagger$ for the projections from $W\times Y\times Z$ and $g_1$ and $g_2$ for the projections $W\times X\times Y\times Z\to W\times X\times Y$ and  $W\times X\times Y\times Z\to W\times Y\times Z$ respectively, we compute
\begin{align*}
(\Gamma'\circ\Gamma)_\ast\alpha&=(q_{WZ})_\ast(q_{WY}^\ast(\Gamma'\circ\Gamma)\cdot q_{YZ}^\ast\alpha)\\
&=(q_{WZ})_\ast((g_2)_\ast g_1^\ast(p_{WX}^\ast\Gamma'\cdot p_{XY}^\ast\Gamma)\cdot q_{YZ}^\ast\alpha)\\
&=(q_{WZ})_\ast (g_2)_\ast(g_1^\ast(p_{WX}^\ast\Gamma'\cdot p_{XY}^\ast\Gamma)\cdot g_2^\ast q_{YZ}^\ast\alpha)\\
&=(q_{WZ})_\ast (g_2)_\ast((g_1^\ast p_{WX}^\ast\Gamma'\cdot g_1^\ast p_{XY}^\ast\Gamma)\cdot g_2^\ast q_{YZ}^\ast\alpha),
\end{align*}
using the projection formula for the action of cycles \Cref{lem-projection formula II} and proper flat base-change for Chow groups freely.

On the other hand, we compute similarly
\[
\Gamma'_\ast(\Gamma_\ast\alpha)=(s_{WZ})_\ast(s_{WX}^\ast\Gamma'\cdot s_{XZ}^\ast(\Gamma_\ast\alpha)),
\]
where we write $s_\dagger$ for the projections from $W\times X\times Z$. Now plugging in $\Gamma_\ast\alpha=(r_{XZ})_\ast(r_{XY}^\ast\Gamma\cdot r_{YZ}^\ast\alpha)$, where $r_\dagger$ is the projection from $X\times Y\times Z$, we obtain
\begin{align*}
\Gamma'_\ast(\Gamma_\ast\alpha)&=(s_{WZ})_\ast(s_{WX}^\ast\Gamma'\cdot (f_1)_\ast f_2^\ast(r_{XY}^\ast\Gamma\cdot r_{YZ}^\ast\alpha))\\
&=(s_{WZ})_\ast(f_1)_\ast(f_1^\ast s_{WX}^\ast\Gamma'\cdot(f_2^\ast r_{XY}^\ast\Gamma\cdot f_2^\ast r_{YZ}^\ast\alpha)),
\end{align*}
where $f_1$ and $f_2$ are the projections $W\times X\times Y\times Z\to W\times X\times Z$ and $W\times X\times Y\times Z\to X\times Y\times Z$ respectively.

By the functorialities of Chow groups and refined unramified cohomology with respect to proper push-forwards and pull-backs between smooth schemes (cf. \Cref{thm-functoriality}), we have the following equalities:
\begin{itemize}
    \item let $d\coloneqq d_X+d_Y$ and $m\coloneqq c+c'$, and then we have $(q_{WZ})_\ast (g_2)_\ast=(s_{WZ})_\ast(f_1)_\ast$, where $H^{p+2m}_{q+m,\nr}(W\times X\times Y\times Z,n+m)\to H^{p+2m-2d}_{q+m-d}(W\times Z,n+m-d)$;
    \item $g_2^\ast q_{YZ}^\ast=f_2^\ast r_{YZ}^\ast\colon H^p_{q,\nr}(Y\times Z,n)\to H^p_{q,\nr}(W\times X\times Y\times Z,n)$;
    \item $g_1^\ast p_{XY}^\ast=f_2^\ast r_{XY}^\ast\colon\CH^c(X\times Y)\to \CH^c(W\times X\times Y\times Z)$;
    \item $g_1^\ast p_{WX}^\ast=f_1^\ast s_{WX}^\ast\colon\CH^{c'}(W\times X)\to \CH^{c'}(W\times X\times Y\times Z)$.
\end{itemize}
So to conclude $(\Gamma'\circ\Gamma)_\ast\alpha=\Gamma'_\ast(\Gamma_\ast\alpha)$, it suffices to show the equality $(\Gamma'\cdot\Gamma)\cdot\alpha=\Gamma'\cdot(\Gamma\cdot\alpha)$ for $\Gamma'\in\CH^{c'}(X)$, $\Gamma\in\CH^c(X)$ and $\alpha\in H^p_{q,\nr}(X,n)$. Using the compatibilities of the pull-back with the cross product \Cref{lem-cross product for pull back} again, we see that we are left to show the equality $(\Gamma'\times\Gamma)\times\alpha=\Gamma'\times(\Gamma\times\alpha)$. This is precisely \Cref{prop-composition-cross}.
\end{proof}

\begin{cor}\label{rem-compatible action}
Let $X$, $Y$ and $Z$ be smooth proper $k$-schemes with dimension $d_X$, $d_Y$ and $d_Z$ respectively. For any $[\Gamma]\in \CH^c(X\times Y)$, we have the following commutative diagram
\begin{equation}
\adjustbox{max width=\textwidth}{
    \xymatrix{H^{p+q-1}_{p-2,\nr}(Y\times Z,n)\ar[r]\ar[d]^{[\Gamma]_*}&E^{p,q}_2(Y\times Z,n)\ar[r]\ar[d]^{[\Gamma]_*}&H^{p+q}_{p,\nr}(Y\times Z,n)\ar[d]^{[\Gamma]_*}\\
    H^{p+q-1+2c-2d_Y}_{p-2+c-d_Y,\nr}(X\times Z,n+c-d_Y)\ar[r]&E^{p+c-d_Y,q+c-d_Y}_2(X\times Z,n+c-d_Y)\ar[r]&H^{p+q+2c-2d_Y}_{p+c-d_Y,\nr}(X\times Z,n+c-d_Y)\nospacepunct{.}
    }
}
\end{equation}
Moreover, this commutative diagram is compatible with composition of correspondences.
\end{cor}
\begin{proof}
As the correspondence action defined in \Cref{new-def-corr-action} is induced by the flat pull-back, exterior product, pull-back and proper push-forward, the result follows from \Cref{prop-exact sequence is compatible with pullback} and \Cref{prop-exact sequence is compatible with cross product}.
\end{proof}
\begin{rem}
\rm{ We remark here that every morphism mentioned in this section (e.g., pull-back between schemes schemes and exterior product) is natural with respect to the restriction maps $H^p_{q,\nr}(X,n)\to H^p_{q-1,\nr}(X,n)$. Therefore, every morphism mentioned is compatible with the long exact sequence (\ref{equ-the canonical long exact sequence}).
    
    }
\end{rem}
We end this section with a more explicit construction of \Cref{new-def-corr-action}, under the extra assumption that we work with smooth \emph{projective} $k$-schemes.
\begin{lem}\label{lem-concrete description of corr actions}
Let $X$, $Y$ and $Z$ be smooth projective $k$-schemes with dimension $d_X$, $d_Y$ and $d_Z$ respectively. Let $[\Gamma]\in \CH^c(X\times Y)$. The bi-additive pairing from \Cref{new-def-corr-action}
\begin{eqnarray*}
    \CH^c(X\times Y)\times H^p_{q,\nr}(Y\times Z,n)&\to& H^{p+2c-2d_Y}_{q+c-d_Y,\nr}(X\times Z,n+c-d_Y)\\
    ([\Gamma],[\alpha])&\mapsto&[\Gamma]_*[\alpha]
\end{eqnarray*}
can be described as follows:
\begin{itemize}
    \item [(i)] for any $[\Gamma]\in \CH^c(X\times Y)$, we can find a representive $\Gamma\in Z^c(X\times Y)$ and closed subsets $|\Gamma|$ and $W$ in $X\times Y$, where $|\Gamma|\coloneqq{\rm Supp }(\Gamma)\subset W$ and $\codim(W)=c'\in\{c,c-1\}$; 
    \item[(ii)] applying \cite[Corollary 6.5]{schreieder2022moving} to the morphism $f\colon W'\times Z\to X\times Y\times Z\to Y\times Z$ (here $W'$ is the disjoint union of the irreducible components of $W$), for any $[\alpha]\in H^p_{q,\nr}(Y\times Z,n)$, there is an open subset $U\subset Y\times Z$ with complement $R$ such that $R$ has pure codimension $q+2$ in $Y\times Z$, $f^{-1}(R)\subset W'$ has locally codimension at least $q+2$ (i.e., $W\times Z$ intersects $X\times R$ properly) and $\alpha\in H^p(U,n)$;
    \item[(iii)] let $W_R\coloneqq  (W\times Z)\cap(X\times R)$ and $R'\coloneqq  p_{XZ}^{XYZ}(W_R)\subset (X\times Z)$, and $U'\coloneqq  (X\times Z)\setminus R'$ be the complements of $R'$, then $W_R$ (\emph{resp.} $R'$) has codimension at least $c'+q+2$ (\emph{resp.} $c'+q+2-d_Y$). Moreover, $F_{q+c-d_Y}(X\times Z)\subset U'$ as $c'\in\{c,c-1\}$;
    \item[(iv)] the element $[\Gamma]_*[\alpha]$ can be represented as the image of $\alpha$ under the morphism $\Gamma(W)_*$
\begin{equation*}
\xymatrix{H^p(U,n)\ar[r]^{{p^{XYZ}_{YZ}}^*}\ar[rrrrddd]^{\Gamma(W)_*}&
H^p(X\times U,n)\ar[rrr]^{\cl^{X\times U}_{(W\times Z)\setminus W_R}(p_{XY}^*\Gamma|_{X\times U})}&&&
H^{p+2c}_{W\times Z}(X\times U,n+c)\ar[d]^{\mathrm{exc}}_{\cong}\\
&&&&H^{p+2c}_{W\times Z}((X\times Y\times Z)\setminus W_R ,n+c)\ar[d]^{\mathrm{restr.}}\\
&&&&H^{p+2c}_{W\times Z}(Y\times U',n+c)\ar[d]^{{p^{XYZ}_{XZ}}_*}\\
&&&&H^{p+2c-2d_Y}(U',n+c-d_Y)\nospacepunct{.}
  }  
\end{equation*}
\end{itemize} 
In particular, if $Z$ is a point, the correspondence action 
\begin{eqnarray*}
    \CH^c(X\times Y)\times H^p_{q,\nr}(Y,n)&\to& H^{p+2c-2d_Y}_{q+c-d_Y,\nr}(X,n+c-d_Y)\\
    ([\Gamma],[\alpha])&\mapsto&[\Gamma]_*[\alpha]
\end{eqnarray*}
coincides with Schreieder's construction in \cite[Corollary 6.8]{schreieder2022moving}.
\end{lem}
\begin{proof}
First, taking $\mathcal{X}\coloneqq X\times Y\times Z\times X\times Y\times Z$, note that $W\times Z$ intersects $X\times R$ properly in $X\times Y\times Z$ if and only if $W\times Z\times X\times R$ intersects the diagonal $\Delta$ properly in $X\times Y\times Z\times X\times Y\times Z$. Therefore, for any $[\alpha]\in H^{p+2c}_{q+c,\nr}(\mathcal{X},n+c)$ such that $[\alpha]$ can be represented by  $\alpha\in H^{p+2c}(\mathcal{X}\setminus (W\times Z\times X\times R),n+c)$, $\Delta^*([\alpha])=[\Delta^*(\alpha)]$ by \Cref{lem-the same pullback}. Let $\pi\colon \mathcal{X}\to X\times Y$ be the projection of $\mathcal{X}$ onto its first two factors, and consider the commutative diagram \ref{diagram-concrete description of corr actions}. Here the second square commutes since for every $\beta\in H^p(X\times U,n)$ we have 
\begin{align*}
    \begin{split}
       &\Delta^*(\cl^{X\times Y\times Z\times X\times U}_{W\times Z\times X\times U}(\pi^*\Gamma|_{X\times Y\times Z\times X\times U})\cup {p^{XYZXYZ}_{XYZ}}^*(\beta))\\
       =&\Delta^*\circ \cl^{X\times Y\times Z\times X\times U}_{W\times Z\times X\times U}(\pi^*\Gamma|_{X\times Y\times Z\times X\times U}) \cup \Delta^*\circ {p^{XYZXYZ}_{XYZ}}^*(\beta)\\
       =&\cl^{X\times U}_{(W\times Z)\setminus W_R}(p_{XY}^*\Gamma|_{X\times U})\cup \beta
    \end{split}
\end{align*}
by \cite[Lemma A.21]{schreieder2022moving}; and the last square commutes by functorialities of cohomology theory with supports. Therefore, by \Cref{new-def-corr-action}, we get a concrete description for projective cases. 

Finally, suppose $Z$ is a point. Then one can see that the construction here is the same as the one in \cite[Section 4]{schreieder2022moving}.

\begin{equation*}\label{diagram-concrete description of corr actions}
\xymatrix{H^p(U,n)\ar[d]^{{p^{XYZ}_{YZ}}^*}\ar@{=}[rr]&&H^p(U,n)\ar[d]^{{p_{YZ}^{XYZ}}^*}\\
H^p(X\times U,n)\ar[d]^{{p_{XYZ}^{XYZXYZ}}^*}\ar@{=}[rr]&&H^p(X\times U,n)\ar[dd]^{\cl^{X\times U}_{(W\times Z)\setminus W_R}(p_{XY}^*\Gamma|_{X\times U})}\\
H^p(X\times Y\times Z\times X\times U,n)\ar[d]^{\cl^{X\times Y\times Z\times X\times U}_{W\times Z\times X\times U}(\pi^*\Gamma|_{X\times Y\times Z\times X\times U})}&\\
H^{p+2c}_{W\times Z\times X\times U}(X\times Y \times Z\times X\times U,n+c)\ar[d]^{\mathrm{exc}}_{\cong}\ar[rr]^{\Delta^*}&&H^{p+2c}_{W\times Z}(X\times U,n+c)\ar[d]^{\cong}_{\rm exc}\\
H^{p+2c}_{W\times Z\times X\times U}(\mathcal{X}\setminus (W\times Z\times X\times R) ,n+c)\ar[d]^{\iota_*}\ar[rr]^{\Delta^*}&&H^{p+2c}_{W\times Z}((X\times Y\times Z)\setminus W_R ,n+c)\ar[d]^{\rm restr.}\\
H^{p+2c}(\mathcal{X}\setminus (W\times Z\times X\times R) ,n+c)\ar[d]^{\Delta^*}&&H^{p+2c}_{W\times Z}(Y\times U',n+c)\ar[d]^{{p^{XYZ}_{XZ}}_*}\\
H^{p+2c}((X\times Y\times Z)\setminus W_R,n+c)\ar[rr]^{ {p^{XYZ}_{XZ}}_*\circ\hspace{1mm} {\rm restr.} }&&H^{p+2c-2d_Y}(U',n+c-d_Y)\nospacepunct{.}
 }\tag{Diagram \ref{lem-concrete description of corr actions}}
\end{equation*}  
\end{proof}

\section{Applications}
In this section, we give some applications related to the aforementioned results. In particular, we generalise some results in \cite[Section 3.4]{kok2023higher}. 
\subsection{Projective Bundle and Blow-Up Formula}\label{application formulas}
Let $X$ be a smooth scheme and $E\to X$ be a vector bundle of rank $r+1$ with the sheaf of sections $\mathcal{E}$. Let $\pi\colon \mathbf{P}(E)\to X$ be the corresponding projective bundle 
\begin{equation*}
    \mathbf{P}(E)=\mathrm{Proj}\bigoplus_k\mathrm{Sym}^k\mathcal{E}^*.
\end{equation*}
The relative line bundle $\mathcal{O}_{\mathbf{P}(E)}(1)$ defines an element of $\CH^1(\mathbf{P}(E))$, say $h$. For any $1\leq n\leq r$, we have $h^n\in \CH^n(\mathbf{P}(E))$. Moreover, we get a morphism $$(H^*)^n\colon H^p_{q,\nr}(X,m)\xrightarrow{\pi^*}H^p_{q,\nr}(\mathbf{P}(E),m)\xrightarrow{h^n\cdot} H^{p+2n}_{q+n,\nr}(\mathbf{P}(E),m+n),$$ using the action of cycles \Cref{def-action of cycles on refined unramified cohomology}.
\begin{prop}[Projective Bundle Formula]\label{prop-projective bundle formula}
We have an isomorphism
\begin{align}\label{equ-projective bundle formula}
    I\colon  \bigoplus_{k=0}^{r} H^{p-2k}_{q-k,\nr}(X,n-k)\xrightarrow{\sim}H^p_{q,\nr}(\P(E),n), \quad
  (\alpha_0,\dots,\alpha_r)\mapsto\sum_{k=0}^r (H^*)^k(\alpha_k),
\end{align}
where we define $(H^*)^0$ as $\pi^*$.
\end{prop}
\begin{proof} 
Consider the following diagram of exact sequences
\begin{equation}\label{equ-diagram of projective bundle formula}
\adjustbox{scale=.8}{
    \xymatrix{\bigoplus\limits_{k=0}^{r}E_2^{q+1-k,p-q-k-1}(X)\ar[r]\ar[d]^{\Tilde{I}}_{\sim}&\bigoplus\limits_{k=0}^{r} H^{p-2k}_{q+1-k,\nr}(X)\ar[r]\ar[d]^{I}&\bigoplus\limits_{k=0}^{r} H^{p-2k}_{q-k,\nr}(X)\ar[r]\ar[d]^{I}&\bigoplus\limits_{k=0}^{r}E_2^{q+2-k,p-q-k-1}(X)\ar[r]\ar[d]^{\Tilde{I}}_{\sim}&\bigoplus\limits_{k=0}^{r} H^{p-2k+1}_{q+2-k,\nr}(X)\ar[d]^{I}\\ 
    E_2^{q+1,p-q-1}(\mathbf{P}(E))\ar[r]&H^p_{q+1,\nr}(\mathbf{P}(E))\ar[r]&H^p_{q,\nr}(\mathbf{P}(E))\ar[r]&E_2^{q+2,p-q-1}(\mathbf{P}(E))\ar[r]&H^{p+1}_{q+2,\nr}(\mathbf{P}(E)),
    }
    }
\end{equation}
where we write $\Tilde{I}$ as the morphism defined in \cite[Scholium 7.3]{viale1997ℋ︁}. Note that (\ref{equ-diagram of projective bundle formula}) is commutative by \Cref{prop-exact sequence is compatible with pullback} and \Cref{prop-exact sequence is compatible with cross product}. 

Suppose that $q+1\geq\dim(X)+r=\dim(\mathbf{P}(E))$, then for $k=0,\dots,r$, we have $q+1-k\geq \dim(X)$, so $H^{p-2k}_{q+1-k,\nr}(X)=H^{p-2k}(X)$ and $H^p_{q+1,\nr}(\mathbf{P}(E))=H^p(\mathbf{P}(E))$. So for this $q+1$, the morphism $I$ is an isomorphism by the projective bundle formula for ordinary cohomology. Using the fact that $\tilde I$ is an isomorphism by \cite[Scholium 7.3]{viale1997ℋ︁} and an application of the five-lemma on the above diagram (\ref{equ-diagram of projective bundle formula}), we see that $I$ is an isomorphism for $q$ as well. We conclude by induction on $q$.
\end{proof}

Following the exact same argument as above but now using \cite[Proposition 7.9]{viale1997ℋ︁}, we also obtain the following \emph{blow-up formula}. Let $Y\subset X$ be a smooth closed subscheme of codimension $r+1$ and let $\tau\colon\tilde X\to X$ be the blow-up of $X$ with centre $Y$. Write $i\colon E\hookrightarrow \tilde X$ for the inclusion of the exceptional divisor, so that $\tau|_E\colon E\to Y$ becomes a projective bundle of rank $r$. 
\begin{prop}[Blow-Up Formula]\label{prop-blow-up bundle}
There is an isomorphism
\begin{align*}
H^p_{q,\nr}(X,n)\oplus\bigoplus_{k=0}^{r-1}H^{p-2k-2}_{q-k-1,\nr}(Y,n-k-1)&\xrightarrow{\sim}H^p_{q,\nr}(\tilde X,n)\colon\\
(\alpha,\alpha_0,\dots,\alpha_{r-1})&\mapsto\tau^\ast\alpha+\sum_{k=0}^{r-1}(\tilde H^k)^\ast(\alpha_k),
\end{align*}
where $(\tilde H^k)^\ast$ is defined as the composition 
\begin{equation*}
 H^p_{q,\nr}(Y,n)\xrightarrow{(H^\ast)^k}H^{p+2k}_{q+k,\nr}(E,n+k)\xrightarrow{i_\ast}H^{p+2k+2}_{q+k+1,\nr}(\tilde X,n+k+1).   
\end{equation*}
\end{prop}
\begin{rem}
 {\rm One can also prove \Cref{prop-projective bundle formula} and \Cref{prop-blow-up bundle} by applying localization sequences \cite[Theorem 1.3]{kok2023higher} and projection formulas \Cref{lem-projection formual for refined} and \Cref{lem-projection formula II}, without using \cite{viale1997ℋ︁}. Furthermore, after defining the Picard group's action on refined unramified cohomology, one can establish a more general projective bundle formula without the need for smoothness conditions.  We leave these for interested readers.
}
\end{rem}
\subsection{Rost Nilpotence Principle}\label{application RNP}
Let $X$ be a smooth projective scheme over a field $k$, then the so-called Rost nilpotence principle (RNP for short) is the following.
\begin{conj}[Rost Nilpotence Principle.]\label{conj-Rost}
Let $X$ be a smooth projective equi-dimensional scheme over a field $k$. Then for every $\Gamma\in \mathrm{End}_k(X)=\CH^{\dim(X)}(X\times X)$ such that $\Gamma_E=0$ for some field extension $E/k$, $\Gamma$ is nilpotent as a correspondence.
\end{conj}
\begin{prop}\label{prop-exact sequence for torsion cycles}
 Let $X$ be a smooth equi-dimensional projective scheme defined over a field $k$. For any prime $\ell$ invertible in $k$, there is an exact sequence
 \begin{equation}\label{equ-exact sequence for torsion cycles}
H^{2i-2}_{i-3,\nr}(X,\mathbb{Q}_{\ell}/\mathbb{Z}_{\ell}(i))\xrightarrow{f} \CH^i(X)[\ell^{\infty}]\xrightarrow{\lambda_X^i}H^{2i-1}(X,\mathbb{Q}_{\ell}/\mathbb{Z}_{\ell}(i))/M^{2i-1}(X)\to H^{2i-1}_{i-2,\nr}(X,\mathbb{Q}_{\ell}/\mathbb{Z}_{\ell}(i)),
\end{equation}
where $\lambda_X^i$ is the defined in \cite{alexandrou2023bloch} for non-algebraically closed fields and is defined in \cite{bloch1979torsion} for algebraically closed fields.
\end{prop}
\begin{proof}
This extends \cite[Corollary 4.20]{kok2023higher}. Noting that if $k=\Bar{k}$, $M^{2i-1}(X)=0$, we only consider the case that $k$ is non-algebraically closed. By \cite[Theorem 1.1]{alexandrou2024two}, we have the following commutative diagram
\begin{equation*}
\adjustbox{width=\textwidth}{
\xymatrix{&\CH^i(X,1;\Q_{\ell})\ar[r]^{\sim}\ar[d]&H^{2i-1}_L(X,\Q_{\ell}(i))\ar[d]&\\
H^{2i-2}_{i-3,\nr}(X,\Q_{\ell}/\Z_{\ell}(i))\ar[r]\ar[rd]^{f}&\CH^i(X,1;\Q_{\ell}/\Z_{\ell})\ar[r]\ar@{->>}[d]&H^{2i-1}_L(X,\Q_{\ell}/\Z_{\ell}(i))\ar[r]\ar@{->>}[d]&H^{2i-1}_{i-2,\nr}(X,\Q_{\ell}/\Z_{\ell}(i))\\
&\CH^i(X)[\ell^{\infty}]\ar[r]^{-\lambda_X^i}&\frac{H^{2i-1}_L(X,\Q_{\ell}/\Z_{\ell}(i))}{M^{2i-1}(X)},&
    }}
\end{equation*}
where the subscript `$L$' denotes the so-called Lichtenbaum (or motivic \'etale) cohomology (cf. \cite[Lecture 10]{mazza2006lecture}). Note that $H^{2i-1}_L(X,\Q_\ell/\Z_\ell(i))\cong H^{2i-1}_{\e t}(X,\Q_\ell/\Z_\ell(i))$ by \cite[Theorem~10.2]{mazza2006lecture}, implying that the middle row in the diagram is exact by \cite[Theorem 1.1]{kok2023higher}. The columns are exact since $$M^{2i-1}(X)=I^{2i-1}(X)=\im(H^{2i-1}_L(X,\Q_{\ell}(i))\to H^{2i-1}_{L}(X,\Q_{\ell}/\Z_{\ell}(i)))$$ by \cite[Theorem 1.1]{alexandrou2024two}. Note that a diagram chasing implies that $M^{2i-1}(X)$ maps to zero in $H^{2i-1}_{i-2,\nr}(X,\Q_{\ell}/\Z_{\ell}(i))$ so that we can naturally define the morphism 
\[
H^{2i-1}(X,\mathbb{Q}_{\ell}/\mathbb{Z}_{\ell}(i))/M^{2i-1}(X)\to H^{2i-1}_{i-2,\nr}(X,\mathbb{Q}_{\ell}/\mathbb{Z}_{\ell}(i)).
\]
The exactness of (\ref{equ-exact sequence for torsion cycles}) follows again from a diagram chasing.
\end{proof}
\begin{lem}\label{lem-compatible on torsion chow groups}
Let $X$, $Y$ and $Z$ be smooth projective equi-dimensional schemes with $d\coloneqq\dim(Y)$, and $[\Gamma]\in \CH^c(X\times Y)$. The exact sequence (\ref{equ-exact sequence for torsion cycles}) is compatible with correspondence action $[\Gamma]_*$.
\end{lem}
\begin{proof}
  Applying \Cref{prop-actions of corr} and recalling the constructions of correspondence action $[\Gamma]_*$ on higher Chow groups, Lichtenbaum cohomology and \'etale cohomology with finite coefficients, it suffices to show $f$ is compatible with correspondence action. In other words, it is enough to show the following diagram commutes
  \begin{equation*}
      \xymatrix{ H^{2i-2}_{i-3,\nr}(Y\times Z,\mathbb{Q}_{\ell}/\mathbb{Z}_{\ell}(i))\ar[r]\ar[d]& \CH^i(Y\times Z,1;M)\ar[d] \\
      H^{2i-2+2c-2d}_{i-3+c-d,\nr}(X\times Z,\mathbb{Q}_{\ell}/\mathbb{Z}_{\ell}(i+c-d))\ar[r]&\CH^{i+c-d}(X\times Z,1;M).
      }
  \end{equation*}
 This is essentially \Cref{prop-commutes with external product}.
\end{proof}
\begin{prop}\label{prop-RNP}
    Let $X$ be a smooth projective equi-dimensional $d$-dimensional scheme defined over a field $k$ of characteristic zero. Let $\Gamma\in {\rm End}_k(X)$ such that $\Gamma_E=0$ for some finite Galois field extension $E/k$. The following are equivalent:
    \begin{itemize}
        \item[(1)] $\Gamma$ is nilpotent as a correspondence;
        \item[(2)] the action of the correspondence $\Gamma$ is nilpotent on the refined unramified cohomology groups $H^{2d-2}_{d-3,\nr}(X\times _kX,\Q_{\ell}/\Z_{\ell}(d))$ for any prime $\ell$, and the nilpotence index of $[\Gamma]_*$ on $H^{2d-2}_{d-3,\nr}(X\times _kX,\Q_{\ell}/\Z_{\ell}(d))$ has a common upper bound for all $\ell$;
        \item[(3)] the action of the correspondence $\Gamma$ is nilpotent on the refined unramified cohomology groups $H^{2d-1}_{d-2,\nr}(X\times _kX,\Z_{\ell}(d))$ for any prime $\ell$, and the nilpotence index of $[\Gamma]_*$ on $H^{2d-1}_{d-2,\nr}(X\times _kX,\Z_{\ell}(d))$ has a common upper bound for all $\ell$.
      \end{itemize}
Here $H^i(*,\Z_{\ell}(n))$ is Jannsen's continuous \'etale cohomology (see \cite{jannsen1988continuous}) with coefficients $(\mu_{\ell^{r}}^{\otimes  n})_r$.
  
\end{prop}
\begin{proof}
 For a standard argument like \cite[Lemma 2.3]{diaz2019rost}, we know that $\Gamma$ is torsion. By the Chinese remainder theorem, we may further assume $\Gamma\in \CH^d(X\times X)[\ell^{\infty}]$ for some prime $\ell$. 
 
 The equivalence of (2) and (3) is given by the Bockstein sequence of refined unramified cohomology by \cite[Proposition 3.16]{kok2023higher}, since the map \[[\Gamma]_*:H^{p}_{q,\nr}(X\times _kX,\Q_{\ell}(d))\to H^{p}_{q,\nr}(X\times _kX,\Q_{\ell}(d))\] is zero if $[\Gamma]$ is torsion.

Next, (1) implies (2). If there exists a positive integer $N$ such that $[\Gamma^{\circ N}]=0\in \CH^d(X\times X)$, for any prime $\ell$, the action induced by $[\Gamma^{\circ N}]_*=[\Gamma]_*^{\circ N}$ is trivial on the refined unramified cohomology groups $H^{2d-2}_{d-3,\nr}(X\times _kX,\Q_{\ell}/\Z_{\ell}(d))$ by \Cref{prop-actions of corr}(3).

Finally, we show that (2) implies (1) conversely. Assume that there exists a positive integer $N$ such that the action $[\Gamma^{\circ N}]_*=[\Gamma]_*^{\circ N}$ is trivial on $H^{2d-2}_{d-3,\nr}(X\times _kX,\Q_{\ell}/\Z_{\ell}(d))$ for any prime~$\ell$. By an argument of \emph{Hochschild--Serre spectral sequence} as in \cite[Remark 3.2]{rosenschon2018rost}, there exists an integer $N'\coloneqq N'(d)$ only depending on $d$ such that $[\Gamma^{\circ N'}]_*=[\Gamma]_*^{\circ N'}$ acts trivially on $H^{2d-1}(X\times X,\Q_{\ell}/\Z_{\ell}(d))$ and $H^{2d-1}(X\times X,\Q_{\ell}/\Z_{\ell}(d))/M^{2d-1}(X\times X)$. Applying \Cref{lem-compatible on torsion chow groups} and \Cref{prop-exact sequence for torsion cycles}, the morphism $[\Gamma^{\circ (N'+N)}]_*:\CH^d(X\times X)[\ell^{\infty}]\to \CH^d(X\times X)[\ell^{\infty}]$ is zero, which implies $[\Gamma^{\circ (N'+N+1)}]=0$.\qedhere

\end{proof}
\begin{rem}
\rm{We would like to remark here that there is actually a simpler perspective on \Cref{prop-RNP}, and we appreciate Schreieder for sharing us. For the `(2) implies (1)' part, arguments as above showing that we can find an integer $M\coloneqq N'(d)+1$ such that $[\Gamma^{\circ M}]\in \ker(\lambda_{X\times X}^d)$. Noting that $H^{2d-2}_{d-3,\nr}(X\times _kX,\Q_{\ell}/\Z_{\ell}(d))\twoheadrightarrow \ker(\lambda_{X\times X}^d)$ by \cite[Theorem 1.6]{schreieder2020refined} and is compatible with correspondence actions as \Cref{lem-compatible on torsion chow groups}. 
}
\end{rem}
\subsection{Varieties with Complete Decomposition of the Diagonal}\label{section-linear vars}
In this subsection, we study smooth projective $k$-varieties that have a \emph{complete decomposition of the diagonal}. 
\begin{define}
We say that a smooth projective $k$-variety $X$ has a complete decomposition of the diagonal if the class of the diagonal $[\Delta_X]\in \CH^d(X\times X)$ can decompose as
\begin{equation}
[\Delta_X]=\sum_{i,j}(\alpha_{i,j}\times \beta_{i,j}), \quad \text{with } \alpha_{i,j}\in \CH^{d-i}(X) \text{ and } \beta_{i,j}\in \CH^{i}(X),
\end{equation}
where $d\coloneqq\dim(X)$ and $i=0,\dots,d$.
\end{define}
A typical example of a class of varieties having this property are \emph{linear varieties}; see \cite{joshua2001algebraic,totaro2014chow}.
\begin{prop}\label{prop-surjective}
Let $X$ be a non-singular projective $k$-variety such that $\Delta_X$ has a complete decomposition. Then the restriction map $H^{p}(X,M(n))\to H^{p}_{q,\nr}(X,M(n))$ is surjective for all $p,q\geq 0$ and all $n$, where  $m$ is invertible in $k$ and $M\coloneqq \Z/m$. In particular, RNP holds for linear varieties in characteristic zero.
\end{prop}
\begin{proof}
For every $x\in H^{p}_{q,\nr}(X,M(n))$, we have $[\Delta_X]_*x=x$ by \cite[Lemma 4.4]{schreieder2022moving}, which means that $\sum_{i,j}(\alpha_{i,j}\times \beta_{i,j})_*x=x$. On the one hand, by \cite[Corollary 6.8(4)]{schreieder2022moving}, for all $i>q$, $(\alpha_{i,j}\times \beta_{i,j})_*x=0$. On the other hand, if $i\leq q$, for such $x$, there exists an open subset $U\subset X$ such that $\codim(X\setminus U)=q+2$ with $x\in H^p(U,M(n))$ and the support of $\alpha_{i,j}$ is contained in $U$, since $q+2+d-i=d+(q-i)+2>d$. Then by \cite[Section 4]{schreieder2022moving}, $(\alpha_{i,j}\times \beta_{i,j})_*x$ comes from $H^p(X,M(n))$. Thus the restriction map is surjective.

   Let $X$ be a smooth projective $d$-dimensional linear variety defined over a characteristic zero field $k$. Then $Y\coloneqq X\times X$ is also a smooth projective linear variety and $\Delta_Y$ has a complete decomposition by \cite[Corollary 4.6]{joshua2001algebraic}. Therefore, $H^p(Y,M(n))\to H^p_{q,\nr}(Y,M(n))$ is always surjective. Suppose that $\Gamma\in \CH^d(Y)$ such that $\Gamma_E=0$ for some finite Galois field extension $E/k$, then there exists $N$ independent of $\ell$ such that the action $\Gamma_*^{\circ N}$ is zero on $H^{2d-2}(Y,\Q_{\ell}/\Z_{\ell}(d))$ from a \emph{Hochschild--Serre spectral sequence} argument as in \cite[p.~427]{rosenschon2018rost}. The surjectivity above now implies that $\Gamma$ acts nilpotently with index $N$ on the refined unramified cohomology group $H^{2d-2}_{d-3,\nr}(X\times _kX,\Q_{\ell}/\Z_{\ell}(d))$ for any prime $\ell$, hence \Cref{prop-RNP} applies and we conclude that RNP holds for $X$.
\end{proof}
\begin{lem}\label{lem-computation for algebraically closed field}
   Let $k$ be an algebraically closed field and $X$ be a non-singular projective $d$-dimensional $k$-variety such that $\Delta_X$ has a complete decomposition. Let $M\coloneqq \Z/m$ where  $m$ is invertible in $k$, then
   \begin{equation*}
  H^p_{q,\nr}(X,M(b))=\left\{\begin{array}{ll}
      H^p(X,M(b)) & \text{if } p\leq 2q \text{ and } p \text{ is even; } \\
      0 & \text{otherwise. }
  \end{array} \right.
\end{equation*}
In particular, the higher cycle class map $\CH^b(X,n;M)\to H^{2b-n}(X,M(b))$ is an isomorphsim for $b,n\geq 0$.
\end{lem}
\begin{proof}
    Let $x$ be an element of $H^p_{q,\nr}(X,M(b))$, then we have  $\sum_{i,j}(\alpha_{i,j}\times \beta_{i,j})_*x=x$. By the projection formula \Cref{lem-projection formual for refined}, $(\alpha_{i,j}\times \beta_{i,j})_*x$ is the image of $x$ under the following compositions of morphisms
    \begin{equation*}
        H^p_{q,\nr}(X,b)\xrightarrow{{\alpha_{i,j}\cdot}}H^{p+2d-2i}_{q+d-i,\nr}(X,b+d-i)\xrightarrow{{\pi_1}_*}H^{p-2i}_{q-i,\nr}(k,b-i)\xrightarrow{{\pi_2}^*}H^{p-2i}_{q-i,\nr}(X,b-i)\xrightarrow{\beta_{i,j}\cdot}H^p_{q,\nr}(X,b),
    \end{equation*}
    where $\pi_1$ (resp. $\pi_2$) is the projection to the first (resp. second) factor. 
    
    Note that $H^{p-2i}_{q-i,\nr}(k,b-i)=0$ unless $q-i\geq 0$ and $p-2i=0$ by \cite[Lemma 5.1]{kok2023higher}. So unless when $2q\geq 2i=p$ and $p$ is even, we have non-trivial $H^p_{q,\nr}(X,M(b))$, and we finish our proof by \cite[Corollary 5.10]{schreieder2020refined} and \cite[Theorem 1.1]{kok2023higher}.
\end{proof}
Applying \cite[Theorem 1.6]{schreieder2020refined}, we can directly obtain that smooth projective complex linear varieties have trivial Griffiths groups.
\begin{rem}
\rm{The conclusion of \Cref{lem-computation for algebraically closed field} is the same as \cite[Theorem 5.2]{joshua2001algebraic}. Moreover, if one follows the same argument of loc.~cit.~and noting that the higher cycle class map $\CH^b(k,n;M)\xrightarrow{\cl} H^{2b-n}(k,M(b))$ is injective, the higher cycle class map $\CH^b(X,n;M)\to H^{2b-n}(X,M(b))$ is always injective for such $X$ and any field $k$. Then the long exact sequence in \cite[Theorem 1.1]{kok2023higher} splits as 
\begin{equation*}
    0\to \CH^b(X,n;M)\to H^{2b-n}(X,M(b))\to H^{2b-n}_{b-n-1,\nr}(X,M(b))\to 0,
\end{equation*}
and we also get \Cref{prop-surjective}. 
    }
\end{rem}
\bibliographystyle{alpha}
\bibliography{ref}

\begin{thebibliography}{MVW06}

\bibitem[Ale24]{alexandrou2024two}
Theodosis Alexandrou.
\newblock Two cycle class maps on torsion cycles.
\newblock {\em International Mathematics Research Notices}, page rnae138, 2024.

\bibitem[AS23]{alexandrou2023bloch}
Theodosis Alexandrou and Stefan Schreieder.
\newblock {On Bloch’s map for torsion cycles over non-closed fields}.
\newblock In {\em Forum of Mathematics, Sigma}, volume~11, page e53. Cambridge University Press, 2023.

\bibitem[Blo79]{bloch1979torsion}
Spencer Bloch.
\newblock {Torsion algebraic cycles and a theorem of Roitman}.
\newblock {\em Compositio mathematica}, 39(1):107--127, 1979.

\bibitem[BO74]{BO74}
Spencer Bloch and Arthur Ogus.
\newblock Gersten's conjecture and the homology of schemes.
\newblock {\em Annales Scientifiques de l'\'{E}cole Normale Sup\'{e}rieure. Quatri\`eme S\'{e}rie}, 7:181--201 (1975), 1974.

\bibitem[CGM05]{MR2110630}
Vladimir Chernousov, Stefan Gille, and Alexander Merkurjev.
\newblock Motivic decomposition of isotropic projective homogeneous varieties.
\newblock {\em Duke Mathematical Journal}, 126(1):137--159, 2005.

\bibitem[CTV12]{CTV12}
Jean-Louis Colliot-Th\'{e}l\`ene and Claire Voisin.
\newblock Cohomologie non ramifi\'{e}e et conjecture de {H}odge enti\`ere.
\newblock {\em Duke Mathematical Journal}, 161(5):735--801, 2012.

\bibitem[Dia19]{diaz2019rost}
H~Anthony Diaz.
\newblock {Rost nilpotence and higher unramified cohomology}.
\newblock {\em Advances in Mathematics}, 355:106770, 2019.

\bibitem[Fuj02]{fujiwara2002proof}
Kazuhiro Fujiwara.
\newblock {A proof of the absolute purity conjecture (after Gabber)}.
\newblock In {\em Algebraic Geometry 2000, Azumino}, volume~36, pages 153--184. Mathematical Society of Japan, 2002.

\bibitem[Ful13]{fulton2013intersection}
William Fulton.
\newblock {\em Intersection theory}, volume~2.
\newblock Springer Science \& Business Media, 2013.

\bibitem[Gil10]{gille2010rost}
Stefan Gille.
\newblock {The Rost nilpotence theorem for geometrically rational surfaces}.
\newblock {\em Inventiones mathematicae}, 181(1), 2010.

\bibitem[Gil14]{gille2014chow}
Stefan Gille.
\newblock {On Chow motives of surfaces}.
\newblock {\em Journal f{\"u}r die reine und angewandte Mathematik (Crelles Journal)}, 2014(686):149--166, 2014.

\bibitem[Gil18]{gille2018rost}
Stefan Gille.
\newblock {On the Rost nilpotence theorem for threefolds}.
\newblock {\em Bulletin of the London Mathematical Society}, 50(1):63--72, 2018.

\bibitem[Jan88]{jannsen1988continuous}
Uwe Jannsen.
\newblock Continuous {\'e}tale cohomology.
\newblock {\em Mathematische Annalen}, 280(2):207--245, 1988.

\bibitem[Jos01]{joshua2001algebraic}
Roy Joshua.
\newblock {Algebraic K-theory and higher Chow groups of linear varieties}.
\newblock {\em Mathematical Proceedings of the Cambridge Philosophical Society}, 130(1):37--60, 2001.

\bibitem[JS03]{JS03}
Uwe Jannsen and Shuji Saito.
\newblock Kato homology of arithmetic schemes and higher class field theory over local fields.
\newblock {\em Documenta Mathematica}, (Extra Vol.):479--538, 2003.
\newblock Kazuya Kato's fiftieth birthday.

\bibitem[KS12]{kerz2012cohomological}
Moritz Kerz and Shuji Saito.
\newblock {Cohomological Hasse principle and motivic cohomology for arithmetic schemes}.
\newblock {\em Publications Math{\'e}matiques de l'IH{\'E}S}, 115:123--183, 2012.

\bibitem[KZ24]{kok2023higher}
Kees Kok and Lin Zhou.
\newblock {Higher Chow groups with finite coefficients and refined unramified cohomology}.
\newblock {\em Advances in Mathematics}, 458:109972, 2024.

\bibitem[Lev98]{levine1998mixed}
Marc Levine.
\newblock {\em Mixed motives}.
\newblock Number~57. American Mathematical Soc., 1998.

\bibitem[MVW06]{mazza2006lecture}
Carlo Mazza, Vladimir Voevodsky, and Charles~A Weibel.
\newblock {\em Lecture notes on motivic cohomology}, volume~2.
\newblock American Mathematical Soc., 2006.

\bibitem[Ros96]{Ros96}
Markus Rost.
\newblock Chow groups with coefficients.
\newblock {\em Documenta Mathematica}, 1:No. 16, 319--393, 1996.

\bibitem[Ros98]{rost1998motive}
Markus Rost.
\newblock {The motive of a Pfister form}.
\newblock {\em preprint}, pages 1--13, 1998.

\bibitem[RS18]{rosenschon2018rost}
Andreas Rosenschon and Anand Sawant.
\newblock Rost nilpotence and {\'e}tale motivic cohomology.
\newblock {\em Advances in Mathematics}, 330:420--432, 2018.

\bibitem[Sch23]{schreieder2020refined}
Stefan Schreieder.
\newblock Refined unramified cohomology of schemes.
\newblock {\em Compositio Mathematica}, 159(7):1466--1530, 2023.

\bibitem[Sch24a]{schreieder2020infinite}
Stefan Schreieder.
\newblock {Infinite torsion in Griffiths groups}.
\newblock {\em {Journal of the European Mathematical Society}}, pages 1--31, 2024.

\bibitem[Sch24b]{schreieder2022moving}
Stefan Schreieder.
\newblock A moving lemma for cohomology with support.
\newblock {\em {\'E}pijournal de G{\'e}om{\'e}trie Alg{\'e}brique}, 2024.

\bibitem[Tot14]{totaro2014chow}
Burt Totaro.
\newblock {Chow groups, Chow cohomology, and linear varieties}.
\newblock In {\em Forum of Mathematics, Sigma}, volume~2, page e17. Cambridge University Press, 2014.

\bibitem[Via97]{viale1997ℋ︁}
Luca~Barbieri Viale.
\newblock $\mathcal{H}$—cohomologies versus algebraic cycles.
\newblock {\em Mathematische Nachrichten}, 184(1):5--57, 1997.

\bibitem[Voe11]{voevodsky2011motivic}
Vladimir Voevodsky.
\newblock {On motivic cohomology with $\Z/\ell$-coefficients}.
\newblock {\em Annals of mathematics}, pages 401--438, 2011.

\end{thebibliography}
\end{document}